\documentclass[a4paper]{amsart}

\usepackage[toc,page]{appendix}

\usepackage{comment}
\usepackage[utf8]{inputenc}
\usepackage{esint}
\usepackage{MnSymbol}
\usepackage{enumitem}
\usepackage[english]{babel}
\usepackage{amsmath}
\usepackage{systeme}
\usepackage{thm-restate}
\usepackage{bbm}
\usepackage[normalem]{ulem}
\usepackage{graphicx}
\usepackage{float}
\usepackage{mathtools}
\usepackage{tikz}
\usepackage{algorithm}
\usepackage{algpseudocode}
\usetikzlibrary{decorations.pathreplacing,arrows.meta,calc}
\usepackage{subcaption}
\usepackage{booktabs}
\usepackage{multirow}

\usepackage[
	colorlinks,
	pdfpagelabels,
	pdfstartview = FitH,
	bookmarksopen = true,
	bookmarksnumbered = true,
	linkcolor = blue,
	plainpages = false,
	hypertexnames = false,
	citecolor = red] {hyperref}

\hypersetup{
	linktoc=page
}

\usepackage[capitalize]{cleveref}
\crefname{enumi}{}{}
\crefname{equation}{}{}

\makeatletter
\def\@tocline#1#2#3#4#5#6#7{\relax
  \ifnum #1>\c@tocdepth 
  \else
    \par \addpenalty\@secpenalty\addvspace{#2}%
    \begingroup \hyphenpenalty\@M
    \@ifempty{#4}{%
      \@tempdima\csname r@tocindent\number#1\endcsname\relax
    }{%
      \@tempdima#4\relax
    }%
    \parindent\z@ \leftskip#3\relax \advance\leftskip\@tempdima\relax
    \rightskip\@pnumwidth plus4em \parfillskip-\@pnumwidth
    #5\leavevmode\hskip-\@tempdima
      \ifcase #1
       \or\or \hskip 1em \or \hskip 2em \else \hskip 3em \fi%
      #6\nobreak\relax
    \dotfill\hbox to\@pnumwidth{\@tocpagenum{#7}}\par
    \nobreak
    \endgroup
  \fi}
\makeatother

\newtheorem{theorem}{Theorem}
\newtheorem*{theorem*}{Theorem}
\newtheorem{proposition}{Proposition}[section]
\newtheorem{lemma}[proposition]{Lemma}
\newtheorem{corollary}[proposition]{Corollary}

\theoremstyle{definition}
\newtheorem{definition}[proposition]{Definition}
\newtheorem{remark}[proposition]{Remark}
\newtheorem{example}[proposition]{Example}
\numberwithin{equation}{section}

\def \R {\mathbb {R}}

\def \E {\mathbb{E}}
\def \P {\mathbb{P}}

\DeclareMathOperator*{\argmax}{arg\,max}

\DeclareMathOperator{\Var}{Var}

\newcommand{\dz}{\, {\rm d} z}

\newcommand{\dy}{\, {\rm d} y}

\newcommand{\bigO}{\mathcal{O}}

\DeclareFontFamily{U}{mathx}{\hyphenchar\font45}
\DeclareFontShape{U}{mathx}{m}{n}{
      <5> <6> <7> <8> <9> <10>
      <10.95> <12> <14.4> <17.28> <20.74> <24.88>
      mathx10
      }{}
\DeclareSymbolFont{mathx}{U}{mathx}{m}{n}
\DeclareFontSubstitution{U}{mathx}{m}{n}
\DeclareMathAccent{\widecheck}{0}{mathx}{"71}

\newcommand{\abs}[1]{\left\lvert#1\right\rvert}

\allowdisplaybreaks

\begin{document}

\title[]{Concentration bounds for intrinsic dimension estimation using Gaussian kernels}
\author[Andersson]{Martin Andersson}
\address{Martin Andersson,
Department of Mathematics,
Uppsala University,
S-751 06 Uppsala,
Sweden}
\email{martin.andersson@math.uu.se}

\begin{abstract}
We prove finite-sample concentration and anti-concentration bounds for dimension estimation using Gaussian kernel sums. Our bounds provide explicit dependence on sample size, bandwidth, and local geometric and distributional parameters, characterizing precisely how regularity conditions influence statistical performance. We also propose a bandwidth selection heuristic using derivative information, supported by numerical experiments.
\end{abstract}

\maketitle

\section{Introduction}
In this paper we address the problem of reliably estimating the \emph{intrinsic dimension} $d$ of a $d$-dimensional Riemannian submanifold $\Omega \subset \R^N$ from i.i.d. observations $X_1,\dots,X_n$ drawn from a distribution with density $p$ and support on $\Omega$, where $d$ is the dimension of the tangent space $T_x\Omega$. This problem arises across data analysis, from characterizing strange attractors \cite{grassberger1983measuring} to understanding generalization in neural networks \cite{ansuini2019intrinsic,birdal2021intrinsic,dupuis2023generalization}. In addition, many dimensionality reduction algorithms require dimension estimates as input \cite{tenenbaum2000global,wang2004adaptive,BerryHarlim2016}. 

Dimension estimation from finite samples is, however, fundamentally 
ambiguous: any point set admits embeddings in manifolds of varying 
dimensions, as in \cref{fig:panel_combo}. To obtain rigorous estimates of $d$, we must therefore impose regularity conditions on the geometry of $\Omega$ and the distribution of our samples.
\begin{figure}[htbp]
    \centering
    \includegraphics[width=0.8\textwidth]{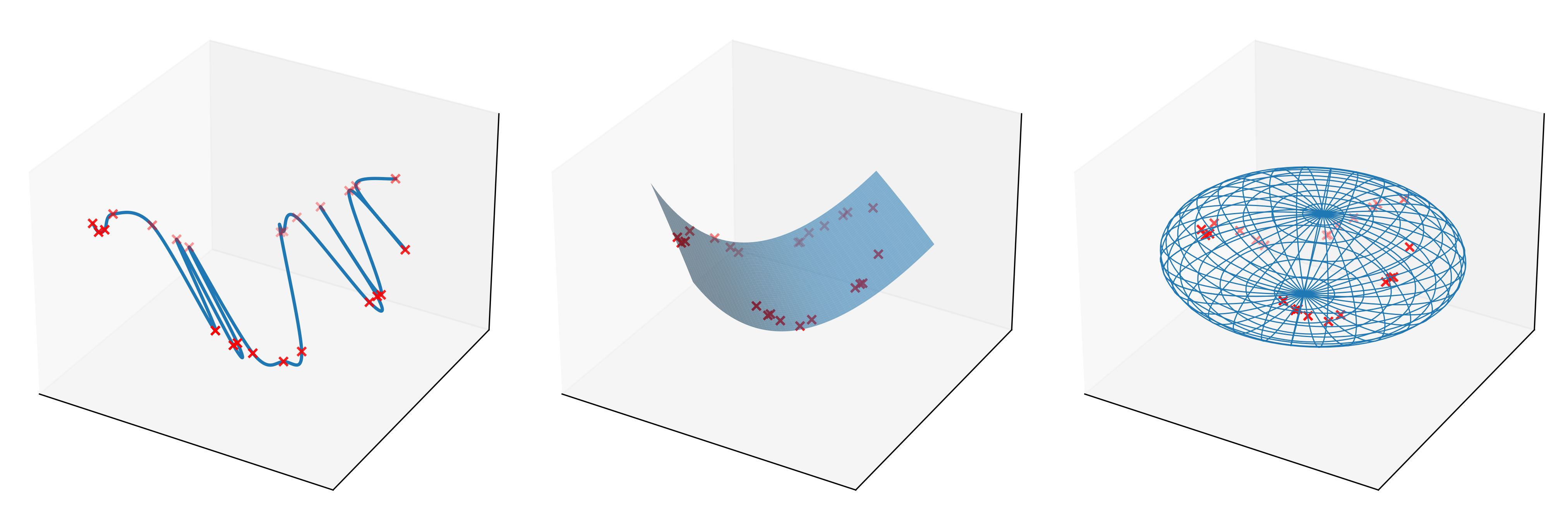}
    \caption{Same point set embedded as a 1D curve, 2D surface, and 3D volume, illustrating the need for regularity assumptions.}
    \label{fig:panel_combo}
\end{figure}
The central question we address is: How do sample size, dimension, and regularity assumptions on the geometry and distribution affect the statistical performance of a Gaussian kernel-based dimension estimator?

Such Gaussian kernel-based estimators were proposed in~\cite{coifman2008}. Their approach uses the Gaussian kernel $K_t(x,y) = e^{-\|x-y\|^2/t}$ and the global kernel sum over all pairs:
\begin{equation*}
    S(t) = \frac{1}{n^2} \sum_{i,j=1}^n K_t(X_i,X_j).
\end{equation*}

For small \emph{bandwidth} $t$ and large $n$, the expected value $\E[K_t(X_1,X_2)]$ scales approximately as $\pi^{d/2}t^{d/2}$, suggesting the dimension can be estimated from the slope of $\log S(t)$ versus $\log t$. Specifically, for appropriately chosen bandwidths $t_1 < t_2$, one estimates
\begin{equation}\label{eq:coifman}
    d \approx 2\frac{\log S(t_2)-\log S(t_1)}{\log(t_2)-\log(t_1)}.
\end{equation}

A central problem here is at what scale to choose the bandwidths $t_1,t_2$: If too small, the effective radius of the kernel does not capture enough points, and if too large the kernel sum saturates and asymptotes to being constantly 1. We need a scale at which the power-law dimensional scaling approximately holds. Moreover, the global estimator in \cref{eq:coifman} provides only a single dimension estimate for the entire dataset, which fails when $\Omega$ has varying local dimension or when $p$ is highly non-uniform.

We therefore study the \emph{local} kernel sum at a point $x \in \Omega$:

\begin{equation*}
    S(x,t) = \frac{1}{n} \sum_{j=1}^n K_t(x,X_j),
\end{equation*}
and the corresponding local dimension estimator
\begin{equation}\label{eq:estimator}
    \hat d(x,t) \coloneq 2\frac{\log S(x,2t)-\log S(x,t)}{\log 2}.
\end{equation}

While it is known that this estimator concentrates to $d$ for sufficiently large $n$ and appropriate choice of $t$, an explicit dependence on sample size, dimension, bandwidth, and the regularity of $\Omega$ and $p$ is, to our knowledge, novel.

On a more practical side, many dimension estimators perform better than worst-case theoretical guarantees suggest, see \cite{binnie2025survey}. When applications do not require absolute rigor, heuristic approaches to hyperparameter selection remain valuable. One of the most common kernels used is the indicator kernel, going back to \cite{grassberger1983measuring}. For this kernel there are varying heuristics for bandwidth (radius in this case) selection, such as using the distance to the $k$-nearest neighbor of $x$, $k$ being some fraction of $n$ \cite{farahmand2007manifold}. One reason the smooth Gaussian kernel can improve on indicator-kernel heuristics is that it enables using derivative information in bandwidth selection. We propose one such method in this paper.

\textbf{Our contributions}:
\begin{enumerate}
    \item We establish concentration bounds (Theorem~\ref{thm:main}) showing $\hat d(x,t)$ 
    concentrates around $d$ with constants having explicit dependence on geometric parameters, bandwidth, regularity of distribution and sample size.
    \item We prove anti-concentration bounds (Theorem~\ref{thm:anticonc:gaussian}, again with explicit dependence on parameters), which upper-bounds the probability $\P(|\hat d - d| \leq \varepsilon)$ that the estimator is within $\varepsilon$ of the true dimension, thereby giving a lower bound on the precision scale $\varepsilon$ at which concentration can occur.
    \item We propose a bandwidth selection heuristic (Algorithm~\ref{alg:bandwidth_selection}) 
    illustrating how derivative information, unavailable for indicator kernels,
    can be utilized. We validate it numerically in \cref{sec:numerical}.
\end{enumerate}
As a practical consequence, our explicit bounds enable finite-sample confidence interval construction for dimension estimates, which existing asymptotic results do not provide.

\textbf{Paper organization}: \cref{sec:related} reviews related work on dimension estimation. In \cref{sec:method} we explain our setup, and in \cref{sec:gaussian,sec:anti} we state and prove our results. \cref{sec:numerical} provides some numerical experiments, and \cref{sec:conclusion} concludes our paper. Supporting lemma proofs are deferred to \cref{sec:app_main,sec:bandwidth_proof}.

\section{Background and Related Work}\label{sec:related}
In the following, we review the most relevant prior work. There are many ways to estimate intrinsic dimension, see \cite{campadelli2015intrinsic,binnie2025survey} for an overview, but one can broadly categorize estimators according to the type of geometric information used \cite{binnie2025survey}.

One such family of estimators can be classified as \emph{tangential estimators}. These use the fact that locally, a smooth $d$-dimensional manifold looks like a graph over its tangent space. This implies that around a point $x\in \Omega$, points should cluster close to a $d$-dimensional affine space. These methods aim to find this affine space and from there estimate the dimension. These include principal component analysis (PCA) and local PCA \cite{fukunaga1971algorithm,fan2010intrinsic}. 

The estimator \cref{eq:estimator} is instead a type of \emph{parametric estimator}. These methods locally approximate the manifold by a flat region with uniform density. In this simplified setting, geometric quantities, such as neighbor counts, distances, and angles, have distributions parameterized by $d$, from which one can build estimators. Examples of this approach include maximum likelihood estimation \cite{levina2004maximum}, angle-and-norm concentration \cite{ceruti2012danco}, and extreme-value-theoretic methods \cite{amsaleg2018extreme}, but there are many others.
 
 A classical example of this is using volume growth. Suppose we have $n$ i.i.d. samples, $X_1, \dots, X_n \in \Omega$, and a point $x \in \Omega$ with density $p(x)$. Define the indicator kernel $K_{r}(x,y) \coloneq \mathbf{1}_{\|x-y\| \leq r}$, let $B_r(x)$ denote the ball of radius $r$ centered at $x$ and $V(\cdot)$ the volume function. Then for large $n$ and small $r$, 
 \begin{align*}
    \frac{1}{n} \sum_i K_{r}(x,X_i) \approx \P(X_1 \in B_r(x)) \approx p(x)V(B_r(x)) = p(x)V(B_1(x))r^d,
 \end{align*}
 and $\frac{\frac{1}{n} \sum_i K_{r_2}(x,X_i)}{\frac{1}{n} \sum_i K_{r_1}(x,X_i)} \approx (r_2/r_1)^d$. This expected scaling can be used to estimate $d$, as was done with the correlation integral, introduced in \cite{grassberger1983measuring},
 \begin{equation}\label{eq:corr_int}
    C(r) = \frac{1}{n^2}\sum_{i,j} K_r(X_i,X_j).
 \end{equation}
 To estimate, one chooses two scales $r_2,r_1$ and examines the scaling of the quotient $C(r_2)/C(r_1)$. However, knowing which scales $r_2, r_1$ to use is both a practical and theoretical challenge.

Instead of examining how the count of neighbors or kernels sums scale with bandwidth, many estimators instead fix the number of neighbors $k$ and use the $k$-nearest neighbor distance as radius. The authors in \cite{farahmand2007manifold} developed a nearest-neighbor ratio estimator with exponential finite sample guarantees. Their estimator takes the form
\begin{equation}
\hat{d}(x) = \frac{\log(2)}{\log(\hat{r}^{(k)}(x)/\hat{r}^{(\lceil k/2 \rceil)}(x))}
\end{equation}
where $\hat{r}^{(k)}(x)$ denotes the distance from $x$ to its $k$-th nearest neighbor. This works similarly as above, that for large $n$, 
\begin{align*}k/n \approx \P(X_1 \in B_{\hat{r}^{(k)}}(x)) \propto (\hat{r}^{(k)})^d.
\end{align*}
 Then $\hat{r}^{(k)}/\hat{r}^{(\lceil k/2 \rceil)} \approx 2^{1/d}$, yielding the dimension estimate above. Once more one faces the problem of choosing hyperparameters, in this case $k$, which is important for performance. Earlier nearest-neighbor approaches include \cite{pettis1979intrinsic}, and a more recent one is the TWO-NN estimator \cite{facco2017estimating}.

 The correlation integral closely resembles the estimator \cref{eq:coifman}, the main difference being the choice of kernel. One can generalize this type of estimator to wider classes of kernels. In \cite{hein2005intrinsic}, they proved that compactly supported kernels on manifolds with curvature and varying distributions asymptotically have the correct scaling. Our work differs in providing finite-sample concentration bounds with explicit constants for the non-compactly supported Gaussian kernel, thereby characterizing the dependence on sample size, geometric regularity and underlying distribution.

The choice of focusing on kernel-based methods here, and specifically the Gaussian kernel, is motivated by their general usefulness in extracting geometric and topological properties of data manifolds. This is a rich and well-established area of research, with some examples being the diffusion maps framework \cite{coifman2006diffusion}, Laplacian eigenmaps \cite{belkin2003laplacian}, kernel PCA \cite{scholkopf1998nonlinear}, and spectral clustering \cite{von2007tutorial}. These all leverage kernel functions to capture intrinsic geometric structures from data samples on manifolds, with the Gaussian kernel being a common choice. One key property of this kernel is its smoothness (exploited in \cref{sec:bandwidth}). Further, kernel-based methods often require dimension estimates for parameter selection (e.g., number of eigenvectors in spectral methods, manifold dimension in diffusion maps), motivating rigorous theoretical analysis of kernel-based dimension estimation.
 
Regarding heuristics in bandwidth selection, \cite{coifman2008} visually inspected the log-log plot to find the linear region. In a refinement, \cite{BerryHarlim2016,BerryHarlim2018} developed this into a more automatic condition, and suggested choosing a bandwidth that maximizes the slope of the log-log plot. More specifically, they create a dyadic partition of an interval $[t_0,t_1]$, evaluate the numerical derivative 
\begin{equation}
    \frac{\log(S(x,2t))-\log(S(x,t))}{\log(2)},
\end{equation}
and choose $t$ to maximize the slope. We try to improve on this in \cref{sec:bandwidth}, and we are not aware of similar approaches to bandwidth selection in the context of dimension estimation.

\section{Methodology and main results}\label{sec:method}
We now formalize our setup and state the main results.
In everything that follows, we assume $X_1,X_2, \dots,X_n \in \Omega \subset \R^N$ are i.i.d.\ random variables with $\kappa$-Lipschitz density $p$ (with respect to the ambient Euclidean metric in $\R^N$) supported on a $d$-dimensional manifold $\Omega \subset \R^N$. We assume access to some fixed point $x \in \Omega$ with $p(x) > 0$ (which could always be taken from the sample itself). We consider the kernel $K_t(x,y) = \exp\left(-\frac{\|x-y\|^2}{t}\right)$, and refer to the parameter $t$ as the bandwidth. We define the normalized kernel sum
\begin{equation}\label{eq:kernel_sum}
    S(x,t) \coloneq \frac{1}{n}\sum_{i=1}^n K_t(x,X_i).
\end{equation}
Let $p(x)$ be the density of a point $x$. Then, for small $t$ and under suitable geometric conditions on the manifold (which we make precise below),
\begin{equation}\label{eq:EK_t}
\E[K_t(x,X)] \approx p(x)\pi^{d/2}t^{d/2}.
\end{equation}
For large $n$, the law of large numbers gives $S(x,t) \approx \E[K_t(x,X)]$, explaining the power-law scaling $S(x,t) \propto t^{d/2}$ and the linear relationship $\log S(x,t) \approx \text{const} + (d/2)\log t$.

From these considerations it is natural to define the estimator
\begin{equation}\label{eq:lloc}
    \hat d(x,t) = 2\frac{\log(S(x,2t))-\log(S(x,t))}{\log(2)}.
\end{equation}

To make our bounds explicit, we need to quantify how 'well-behaved' the manifold $\Omega$ is. We do this through three local parameters: $L$ controls how much $\Omega$ curves away from its tangent space, $M$ controls volume distortion, and $r$ ensures $\Omega$ doesn't fold 
back on itself. Formally:

\begin{definition}[Local $(L,M,r)$-regularity]\label{def:LMr-regular}
Let $\Omega \subset \mathbb{R}^N$ be a smooth $d$-dimensional submanifold. We say $\Omega$ is \emph{locally $(L,M,r)$-regular at $x \in \Omega$} if there exist constants $L,M \geq 0$ and $r > 0$ such that in the ball $B_r(x)$, there exists a smooth orthogonal projection map $\pi: B_r(x) \cap \Omega \to B_r(x) \cap (x+ T_x\Omega)$ satisfying:

\begin{enumerate}[label=(\roman*)]
    \item For all $y \in B_r(x) \cap \Omega$ with $z = \pi(y) - x \in T_x\Omega$ and $h_z = \pi(y) - y$:
    \begin{equation*}
        \|x-y\|^2 = \|z\|^2 + \|h_z\|^2, \quad \|h_z\| \leq L\|z\|^2
    \end{equation*}
    
    \item The volume distortion satisfies $(1-M\|z\|^2) \leq V(y) \leq (1+M\|z\|^2)$ where $V(y)$ is the Jacobian determinant of $\pi^{-1}$.
\end{enumerate}
\end{definition}

\begin{remark}
The graph parametrization implies $|V(y)-1| = \bigO(L^2\|z\|^2)$, so one could absorb $M$ into $L^2$. We keep them separate because curvature and volume distortion enter the bounds of $t$ (see~\cref{lem:mult}) in differing ways.
\end{remark}

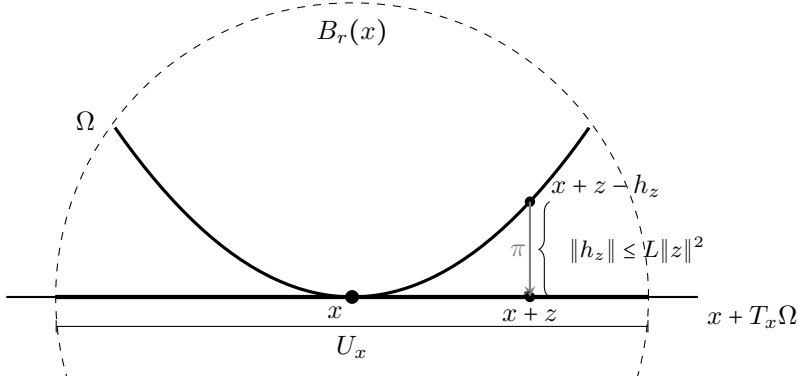
\begin{figure}[ht]
\centering
\begin{tikzpicture}[scale=1.3, >=stealth]
    \clip (-3.7,-0.8) rectangle (5.5,3.3);
    
    \coordinate (x) at (0,0);
    
    \draw[black, dashed] (x) circle (3);
    \node[black] at ({2.7*cos(90)},{2.7*sin(90)}) {$B_r(x)$};
    
    \draw[thick] (-3.5,0) -- (3.5,0);
    
    \draw[ultra thick] (-3,0) -- (3,0);
    \draw[|-|, thin] (-3,-0.3) -- (3,-0.3);
    \node[below] at (0,-0.3) {$U_x$};
    
    \draw[very thick] plot[domain=-2.4:2.4, samples=50] 
        (\x, {0.3*\x*\x});
    \node at (-2.7,1.8) {$\Omega$};
    
    \coordinate (xz) at (1.8,0);
    \coordinate (p) at (1.8,0.972); 
    
    \fill (x) circle (2pt);
    \node[below left] at (x) {$x$};
    
    \fill (xz) circle (1.5pt);
    \fill (p) circle (1.5pt);
    
    \draw[->, thick, gray] (p) -- (xz);
    \node[left, gray] at (1.86,0.48) {$\pi$};
    
    \node[below] at (1.8,-0.015) {$x+z$};
    \node[above right] at (1.92,0.92) {$x+z-h_z$};
    
    \node[below right] at (3.5,0) {$x+T_x\Omega$};
    
    \draw[decorate,decoration={brace,amplitude=4pt}] 
        (2.0,0) -- (2.0,0.972);
    \node[right,font=\small] at (2.1,0.48) {$\|h_z\| \leq L\|z\|^2$};
    
    \draw (1.75,0) -- (1.75,0.05) -- (1.8,0.05);
\end{tikzpicture}
\caption{$(L,M,r)$-regular at $x$: $\Omega \cap B_r(x)$ projects orthogonally onto $U_x = x+T_x\Omega \cap B_r(x)$.}
\label{fig:LMr}
\end{figure}

This particular formulation is also useful for our analysis, as we frequently change variables via the projection $\pi$ and need explicit control over the resulting distortion.

Here follows a simple example where we calculate these parameters.

\begin{example}[Circle of radius $R$]\label{ex:circle}
Let $\Omega = \{(x,y) \in \mathbb{R}^2 : x^2 + y^2 = R^2\}$ be the circle of radius $R$.

Consider the point $x = (R,0)$. The tangent space at $x$ is vertical: $T_x\Omega = \text{span}\{(0,1)\}$, and the normal space is horizontal: $(T_x\Omega)^\perp = \text{span}\{(1,0)\}$.

The orthogonal projection is $\pi_x(a,b) = (R,b)$. The inverse is
\begin{equation*}
    \pi_x^{-1}(R, t) = (\sqrt{R^2-t^2}, t) = (R,t) + (\sqrt{R^2-t^2}-R, 0),
\end{equation*}
for $|t|\leq R/2$.

We expand $\sqrt{R^2-t^2} = R\sqrt{1-t^2/R^2} = R - \frac{t^2}{2R} + \bigO(t^4/R^3)$, which gives
\begin{equation*}
h_z = (R,t) - \left(\sqrt{R^2-t^2},t\right) = \left(R - \sqrt{R^2-t^2}, 0\right) = \left(\frac{t^2}{2R}, 0\right) + \bigO(t^4/R^3),
\end{equation*}
which satisfies $\|h_z\| \le \frac{1}{R}\|z\|^2$ for $z = (0,t)$.

The arc length element is $ds = \frac{R\,dt}{\sqrt{R^2-t^2}} = 1 + \frac{t^2}{2R^2} + \bigO(t^4/R^4)$, giving Jacobian $V = 1 + \frac{t^2}{2R^2} + \bigO(t^4/R^4)$.

Thus, the conditions of \cref{def:LMr-regular} hold with $L = \frac{1}{R}$, $M = \frac{1}{R^2}$ and $r=R/2$.
\end{example}

\begin{remark}
For a manifold with reach $\tau_\Omega$ \cite{aamari2019estimating}, 
one can take $L = \bigO(1/\tau_\Omega)$, $M = \bigO(1/\tau_\Omega^2)$, and $r \lesssim \tau_\Omega$, so the regularity parameters are controlled by the reach alone. See \cite{aamari2019estimating} on the problem of estimating reach from data.
\end{remark}

In the rest of this paper, we assume that $\Omega$ is locally $(L,M,r)$-regular at each point $x \in \Omega$ where we estimate the dimension.

\begin{table}[ht]
\centering
\caption{Summary of notation. Symbols from later sections are included for forward reference.}
\label{tab:notation}
\begin{tabular}{@{}ll@{}}
\toprule
Symbol & Description \\
\midrule
\multicolumn{2}{@{}l}{\textit{Geometry and distribution}} \\[2pt]
$d,\; N$ & Intrinsic and ambient dimension \\
$\Omega \subset \R^N$ & $d$-dimensional submanifold \\
$n$ & Number of i.i.d.\ samples $X_1,\dots,X_n \in \Omega$ \\
$p,\;\kappa$ & Density on $\Omega$ and its Lipschitz constant \\
$L,\;M,\;r$ & Local regularity parameters (\cref{def:LMr-regular}) \\[4pt]
\multicolumn{2}{@{}l}{\textit{Estimator}} \\[2pt]
$K_t(x,y)$ & Gaussian kernel $\exp(-\|x-y\|^2/t)$ \\
$t$ & Bandwidth parameter \\
$S(x,t)$ & Normalized kernel sum $\frac{1}{n}\sum_i K_t(x,X_i)$ \\
$\hat d(x,t)$ & Dimension estimator \cref{eq:lloc} \\
$\varepsilon$ & Target precision for $|\hat d - d|$ \\[4pt]
\multicolumn{2}{@{}l}{\textit{Concentration bounds (\cref{sec:gaussian})}} \\[2pt]
$\gamma$ & Exponent for integration radius $r_0 = t^\gamma$, $\gamma \in (1/4,1/2)$ \\
$\eta$ & Multiplicative error tolerance in moment bounds \\
$c$ & Concentration margin, controls strength via $\eta$ \\
$P_t$ & Idealized expected kernel sum $2^{d/2}p(x)\pi^{d/2}t^{d/2} \approx \E[K_{2t}]$ \\
$\alpha(\hat t)$ & Gaussian tail fraction, controlling incomplete integration over $B_{r_0}(x)$ \\
$\beta_{d,\eta}$ & Gaussian concentration threshold, $\sqrt{2}(\sqrt{d}+\sqrt{2\log(10/\eta)})$ \\
$W^\pm_i$ & $K_{2t}(x,X_i) - 2^{(d \pm \varepsilon)/2}K_t(x,X_i)$, concentration random variables \\
$\varepsilon^*$ & Critical resolution $\sim 1/\!\sqrt{nP_t}$ \\
$t_0$ & Bandwidth threshold for multiplicative bounds (\cref{lem:mult}) \\[4pt]
\multicolumn{2}{@{}l}{\textit{Anti-concentration bounds (\cref{sec:anti})}} \\[2pt]
$Y_\pm$ & $-W^\pm_i$, anti-concentration random variables \\
$\Gamma_\pm$ & Variance constants (\cref{lem:anticoncentration_gaussian}) \\
$\Delta_\pm$ & Third moment constants (\cref{lem:anticoncentration_gaussian}) \\
$\eta^*_\pm$ & Adjusted tolerance parameters \\
\bottomrule
\end{tabular}
\end{table}

\subsection{Concentration bounds}\label{sec:gaussian}
With the setup of the previous section we proceed with our main theoretical results. The first is a concentration inequality, and the main tool we use is Bernstein's inequality (\cref{thm:bernstein} in the appendix).

The major work of this section consists in finding bounds on the expectation and variance of certain random variables whose structure emerges from decomposing the tail probabilities $\P\left(\hat d - d \geq \varepsilon \right)$ and $\P\left(\hat d - d \leq -\varepsilon \right)$, where $\varepsilon > 0$ is our target precision. If we define
\begin{equation*}
    \begin{split}
W^+_i &= K_{2t}(x,X_i) - 2^{(d + \varepsilon)/2}K_t(x,X_i) \\
W^-_i &= K_{2t}(x,X_i) - 2^{(d - \varepsilon)/2}K_t(x,X_i). 
    \end{split}
\end{equation*}
Then, by direct calculation,
\begin{align}
    \P\left(\hat d - d\geq \varepsilon\right) 
    &= \P\left(\frac{2}{\log 2}\log\frac{S(x,2t)}{S(x,t)}-d \geq \varepsilon\right) \nonumber \\
    &= \P\left(S(x,2t)-2^{(d+\varepsilon)/2}S(x,t) \geq 0 \right) \nonumber \\
    &= \P\left(\frac{1}{n}\sum_{i=1}^n W^+_i - \E[W^+_i] \geq -\E[W^+_i]\right) \label{eq:gauss:upper_i}
\end{align} 
and similarly for the lower bound, giving us
\begin{equation}\label{eq:gauss:lower_i}
     \P\left(\hat d - d \leq -\varepsilon\right) = \P\left(\frac{1}{n}\sum_{i=1}^n W^-_i - \E[W^-_i] \leq -\E[W^-_i]\right).
\end{equation}
In this form, the probabilities we need to bound are linear in $X_1,\dots,X_n.$ From the preceding, we see that we need to understand 
\begin{align}
W^+ &= K_{2t}(x,X) - 2^{(d + \varepsilon)/2}K_t(x,X) \label{eq:wplus} \\
\intertext{and}
W^- &= K_{2t}(x,X) - 2^{(d - \varepsilon)/2}K_t(x,X) \label{eq:wminus},
\end{align}
where $X \sim p$.

In order to apply Bernstein's inequality we need to control expectation and variance of $W^+$ and $W^-$. We do this by controlling the moments of $K_{t}$ and $K_{2t}$. To achieve the latter, a useful property of Gaussian kernels is that products satisfy
\begin{equation}\label{eq:mom}
    K_a(x,y) K_b(x,y) = K_{ab/(a+b)}(x,y),
\end{equation}
and in particular $K_a(x,y)^m = K_{a/m}(x,y)$. 

The following lemma addresses how the geometric parameters $L,M,r$, Lipschitz constant $\kappa$, and chosen bandwidth affect deviation from the ideal scaling of \cref{eq:EK_t}:
\begin{lemma}[General Moment Bounds]\label{lem:general}
    Fix $x \in \Omega$ with $\mathrm{dist}(x,\partial \Omega) > 0$ (or $\partial\Omega = \emptyset$). Let $\gamma \in \left(\frac{1}{4},\frac{1}{2}\right)$, denote $r_0 =  t^\gamma$ and assume $X \sim p$. 
    Then for any $t > 0$ satisfying $r_0 = t^\gamma \leq \min\{r,\mathrm{dist}(x,\partial \Omega)\}$,
    \begin{align}\label{eq:mom:general:lb}
        \E[K_{ t}(x,X)] \geq \left(1-\frac{2\kappa r_0}{p(x)}\right)(1-Mr_0^2)(1-\alpha(t))e^{-L^2r_0^4/ t}p(x)\pi^{d/2}t^{d/2}
    \end{align}
    where $\alpha(t) = 1 - \frac{\int_{\|z\|\leq r_0}e^{-\|z\|^2/t}\,dz}{\pi^{d/2}t^{d/2}}$, and
    \begin{align}\label{eq:mom:general:ub}
        \E[K_{ t}(x,X)] \leq\left(1+\frac{2\kappa r_0}{p(x)}\right)(1+Mr_0^2)p(x)\pi^{d/2}t^{d/2} + \tau( t),
    \end{align}
    where $\tau( t) = e^{-r_0^2/ t}$.
\end{lemma}

\begin{remark}\label{rem:gamma}
An optimal choice of $\gamma$ depends on the parameters involved. In our analysis, $\tau(t)$ is a tail term, and large $\gamma$ implies a small tail term but at the same time it increases error accumulation from other factors. 
\end{remark}

To achieve more manageable expressions in our later bounds, we introduce an auxiliary variable $0<\eta<1/2$,  which allows us to express our bounds in a more simple form. This introduces somewhat stricter conditions on $t$.

We need the bounds to hold uniformly over several moments of $K_t$, which is equivalent to bounds of $K_t$ over several values of $t$. In the following lemma, we find a threshold $t_0$ that gives such uniform control.

\begin{lemma}[Multiplicative Bounds]\label{lem:mult}
Assume the conditions of \cref{lem:general} and let $\eta\in(0,1/2)$. There exists a threshold $t_0 = t_0(\eta,L,M,\kappa,\gamma,d,p(x)) > 0$, given explicitly in \cref{sec:app_main}, such that for any $t\le t_0$ and $\hat t\in\mathcal H=\{t,t/2,2t,2t/3\}$,
\begin{equation}\label{eq:mom:mult}
(1-\eta)\,p(x)\pi^{d/2}\hat t^{d/2}
\le
\E\!\left[K_{\hat t}(x,X)\right]
\le
(1+\eta)\,p(x)\pi^{d/2}\hat t^{d/2}.
\end{equation}
\end{lemma}
\begin{remark}
  The threshold $t_0$ is the largest bandwidth at which five constraints are simultaneously satisfied, controlling curvature ($L$), volume distortion ($M$), density variation ($\kappa$), and the tail error ($\gamma,\alpha$). Flatter geometry and smoother density permit larger $t_0$, improving concentration. See \cref{sec:app_main} for the explicit formula and computational details.
  \end{remark}
Then we have a corollary which gives us the necessary control on $\E \left[W^\pm\right]$ and $\Var \left( W^\pm \right)$. Before stating the result, let us for convenience define
\begin{equation}\label{eq:P_t}
    P_t \coloneqq 2^{d/2}p(x)\pi^{d/2}t^{d/2} 
    \approx \mathbb{E}[K_{2t}(x,X)],
\end{equation}
Thus,
\begin{corollary}[Properties of $W^+$ and $W^-$]\label{cor:w}
    Assume the conditions of \cref{lem:mult}. Further, let $0 < c <1$ and $\eta = c\frac{2^{\varepsilon/2}-1}{1+2^{\varepsilon/2}}$. Then
    \begin{align*}
       \E[W^+] &\leq (1-c)(1-2^{\varepsilon/2})P_t < 0, \\
        \E[W^-] &\geq (1-c)(1-2^{-\varepsilon/2})P_t>0.
    \end{align*}
    Furthermore,
    \begin{equation*}
        \begin{split}
        \Var(W^+) &\leq 2^{\varepsilon+3}P_t \\
        \Var(W^-) &\leq  2^{-\varepsilon+3} P_t.
        \end{split}
    \end{equation*}
\end{corollary}
\begin{remark}
The condition on $\eta$ ensures $\E[W^\pm]$ have the correct sign with sufficient margin (controlled by $c$) to apply Bernstein's inequality. Smaller $c$ gives stricter conditions on $t$.
\end{remark}

Now we can state and prove the main result of this section:
\begin{theorem}[Concentration of Dimension Estimator]\label{thm:main}
Under the conditions of \cref{lem:mult,cor:w}, we have
\begin{equation*}
    \P(\hat d - d\geq \varepsilon) \leq \exp\left(\frac{-nc_+^2 P_t}{2^{4+\varepsilon} + \frac{2}{3}M_{d,\varepsilon}c_+}\right)
\end{equation*}
and
\begin{equation*}
    \P(\hat d - d\leq -\varepsilon) \leq  \exp\left(\frac{-nc_-^2 P_t}{2^{4-\varepsilon} + \frac{2}{3}M_{d,-\varepsilon}c_-}\right),
\end{equation*}
where $M_{d,\pm\varepsilon} = 1+2^{(d\pm\varepsilon)/2}$, $c_+=(1-c)(2^{\varepsilon/2}-1)$ and  $c_-=(1-c)(1-2^{-\varepsilon/2})$.
\end{theorem}
Using \cref{cor:w}, the proof of \cref{thm:main} is straightforward:
\begin{proof}[Proof of \cref{thm:main}]
    First we recall the decompositions \cref{eq:gauss:upper_i,eq:gauss:lower_i} and variables \cref{eq:wplus,eq:wminus}. Then we observe that
    \begin{align*}
        -2^{(d+\varepsilon)/2} &\leq K_{2t}(x,X)-2^{(d+\varepsilon)/2}K_t(x,X) \leq 1, \\
        -2^{(d-\varepsilon)/2} &\leq K_{2t}(x,X)-2^{(d-\varepsilon)/2}K_t(x,X) \leq 1,
    \end{align*}
    and thus $|W^{\pm} - \E[W^{\pm}]| \leq 1+2^{(d\pm \varepsilon)/2} = M_{d,\pm\varepsilon}$.

    Using our estimates from \cref{cor:w}, a direct application of Bernstein's inequality to \cref{eq:gauss:upper_i} gives
    \begin{align*}
        \P\left(\hat d - d\geq \varepsilon\right) &\leq \exp\left(\frac{-n\left(c_+ P_t\right)^2}{2\left(2^{\varepsilon+3}P_t+M_{d,\varepsilon}\left(c_+ P_t\right)/3\right)}\right)  \\
        & = \exp\left(\frac{-nc_+^2 P_t}{2^{4+\varepsilon} + \frac{2}{3}M_{d,\varepsilon}c_+}\right),
    \end{align*}
    and similarly for $-\varepsilon$ in \cref{eq:gauss:lower_i},
    \begin{equation*}
        \begin{split}
        \P\left(\hat d - d\leq -\varepsilon\right) &\leq \exp\left(\frac{-n \left(c_- P_t\right)^2}{2\left(2^{-\varepsilon+3}P_t+M_{d,-\varepsilon}\left(c_- P_t\right)/3\right)}\right)  \\
        & = \exp\left(\frac{-nc_-^2 P_t}{2^{4-\varepsilon} + \frac{2}{3}M_{d,-\varepsilon}c_-}\right).
        \end{split}
    \end{equation*}
\end{proof}

\begin{remark}
    Computing $t_0$ of \cref{lem:mult} requires knowing $L,M,\kappa,d,p(x)$. In practice, these must be estimated from data or conservatively bounded. Our contribution is showing \emph{how} geometry is affected by these parameters.
\end{remark}

\begin{remark}[Flat manifolds]
When $\Omega$ has zero curvature ($L = M = 0$ in \cref{def:LMr-regular}), the curvature and volume distortion constraints in $t_0$ (see \cref{eq:t0_explicit}) vanish, leaving only the tail decay and density variation. For uniform density ($\kappa = 0$), only the tail decay conditions apply. This allows larger bandwidths than in the curved case, improving the concentration rate. This is consistent with the improved concentration observed for flat manifolds in \cref{fig:concentration_ball} compared to curved manifolds in \cref{fig:concentration_sphere}.
\end{remark}

\begin{remark}[Small-$\varepsilon$ behavior]\label{rem:small_eps_conc}
For small $\varepsilon$, both $c_\pm \approx (1-c)\frac{\varepsilon\log 2}{2}$
and the denominators in \cref{thm:main} are dominated by their
constant terms, giving
\[
\P(|\hat d(x,t) - d| \geq \varepsilon)
  \leq 2\exp\!\left(-\frac{(1-c)^2(\log 2)^2}{2^6}
       \,nP_t\,\varepsilon^2\,(1+O(\varepsilon))\right).
\]
Together with \cref{rem:small_eps}, this identifies
$\varepsilon^* \sim 1/\sqrt{nP_t}$ as the critical resolution.
When $\varepsilon \gg \varepsilon^*$ the estimator concentrates
exponentially, while for $\varepsilon \ll \varepsilon^*$
reliable estimation at precision $\varepsilon$ becomes impossible.
\end{remark}

\subsection{Anti-concentration}\label{sec:anti}
\cref{thm:main} provides an upper bound on the probability that $\hat d$ deviates from $d$. In this section we prove a complementary result that establishes when our estimator \emph{cannot} reliably distinguish $d$ from nearby values. Specifically, we seek an upper bound on
\begin{equation}\label{eq:anti}
    \P\left(|\hat d - d| \leq \varepsilon\right).
\end{equation}

We work with the same theoretical setup and assumptions made on $\Omega$ and $X_1,\dots, X_n$ as in \cref{sec:method}. While dimension estimators can often output noninteger (fractional) dimensions, our anti-concentration proofs require integer $d$. This is because they rely on manifold structure; they specifically require integration over the $d$-dimensional tangent space $T_x\Omega \cong \R^d$.

In order to bound \cref{eq:anti}, we use normal approximation through the Berry-Esseen theorem (\cref{thm:BE-iid} in the appendix), which controls the approximation error in terms of the mean $\mu$, variance $\sigma^2$, and third absolute moment $\rho$ of certain summands. Similarly to the proof of \cref{thm:main}, the main work thus consists in finding bounds on these quantities for appropriate random variables.

First we rewrite $ \P\left(|\hat d - d| \leq \varepsilon\right)$,
\begin{equation}\label{eq:gaussian:anti_decomp}
    \begin{split}
    \P(|\hat d - d| \leq \varepsilon) =& \P\left(\abs{\frac{2}{\log(2)} \log\left(\frac{S(x,2t)}{S(x,t)}\right) - d} \leq \varepsilon\right) \\
    =& \P\left(\frac{S(x,2t)}{S(x,t)} \in \left[2^{d/2-\varepsilon/2}, 2^{d/2+\varepsilon/2}\right]\right) \\
    =& \P\left(\frac{S(x,2t)}{S(x,t)} \geq 2^{d/2-\varepsilon/2} \right) - \P\left(\frac{S(x,2t)}{S(x,t)} \geq 2^{d/2+\varepsilon/2} \right) \\
    =& \P\left(\frac{S(x,2t)}{S(x,t)} \geq 2^{d/2-\varepsilon/2} \right) - \left(1 -\P\left(\frac{S(x,2t)}{S(x,t)} \leq 2^{d/2+\varepsilon/2} \right)\right) \\
    =&  \P\left(2^{(d-\varepsilon)/2} S(x,t) - S(x,2t) \leq 0 \right)
 \\
 & + \P\left(2^{(d+\varepsilon)/2} S(x,t) - S(x,2t) \geq 0 \right) - 1.
    \end{split}
\end{equation}
Similar to the previous section, we need to understand the random variables
\begin{equation}\label{eq:Y_plus_minus}
    \begin{split}
Y_+ &= 2^{(d+\varepsilon)/2} K_t(x,X) - K_{2t}(x,X) \\
Y_- &=2^{(d-\varepsilon)/2} K_t(x,X) - K_{2t}(x,X), 
    \end{split}
\end{equation}
where $X \sim p$.
To prove the main theorem of this section, the strategy is to use normal approximation on the decomposition \cref{eq:gaussian:anti_decomp}, and to do this we need bounds on $\mu, \sigma^2$ and $\rho$ for the random variables above, $Y_+,Y_-$.
\begin{lemma}\label{lem:anticoncentration_gaussian}
    Let $Y_+, Y_-$ be as in \cref{eq:Y_plus_minus}. Define
    \begin{align*}
        \Gamma_\pm &= 1-2^{1\mp\varepsilon/2}\left(\frac{2}{3}\right)^{d/2}+2^{-d/2\mp\varepsilon}, \\
        \Delta_\pm &= \left(\frac{2}{3}\right)^{d/2} + 3 \cdot 2^{d/2\mp\varepsilon/2}5^{-d/2} + 3 \cdot 2^{\mp\varepsilon-d} + 2^{\mp 3\varepsilon/2-d/2} 3^{-d/2}.
    \end{align*}
    Assume $\Gamma_+ > 0$ and let $c \in (0,1)$ satisfy
    \begin{equation}
        \eta \coloneqq c\frac{2^{\varepsilon/2}-1}{1+2^{\varepsilon/2}} < \frac{\Gamma_+}{1+2^{1-\varepsilon/2}(2/3)^{d/2}+2^{-d/2-\varepsilon}},
    \end{equation}
    and define $\eta^{*}_\pm = \eta\frac{1+2^{1\mp\varepsilon/2}(2/3)^{d/2}+2^{-d/2\mp\varepsilon}}{\Gamma_\pm}$.

    Then, for $t$ sufficiently small:
    \begin{align*}
        \E[Y_+] &\leq (2^{\varepsilon/2}-1)(1+c)P_t, \\
        \E[-Y_-] &\leq (1-2^{-\varepsilon/2})(1+c)P_t, \\
        \Var(Y_\pm) &\geq 2^{\pm\varepsilon}\Gamma_\pm(1-\eta^{*}_\pm)P_t + \bigO(P_t^2), \\
        \E[|Y_\pm -\E[Y_\pm]|^3] &\leq (1+\eta)\Delta_\pm 2^{d/2\pm 3\varepsilon/2}P_t + \bigO(P_t^2).
    \end{align*}
    The condition on $\eta$ ensures $\eta^*_\pm < 1$.
\end{lemma}
\begin{remark}
The condition ``$t$ sufficiently small'' means the multiplicative bounds of \cref{lem:mult} hold for the extended set $\hat{t} \in \{t/3, 2t/5, t/2, 2t/3, t, 2t\}$. The third moment bounds require $K_t^3 = K_{t/3}$ and $K_t^2 K_{2t} = K_{2t/5}$, which are not in the original set $\mathcal{H}$. This will potentially yield a slightly smaller threshold $t_0$ than in \cref{lem:mult}.
\end{remark}

Now we have all the bounds we need, and can apply Berry-Esseen:

\begin{corollary}[Berry-Esseen bound]\label{cor:berry_esseen} 
    Define 
    $Y_{i,\pm} = 2^{(d\pm\varepsilon)/2} K_t(x,X_i) - K_{2t}(x,X_i)$ for $i=1,\ldots,n$.
    Under the assumptions of \cref{lem:anticoncentration_gaussian}, 
    \begin{multline*}
        \sup_{x\in\mathbb{R}}\left|\P\left(\frac{\sum_{i=1}^n (Y_{i,\pm} - \E[Y_{i,\pm}])}{\sqrt{n\Var(Y_{i,\pm})}} \leq x\right) - \Phi(x)\right| 
        \leq \frac{C(1+\eta)\Delta_{\pm}2^{d/2}}{\sqrt{n}\sqrt{P_t}(1-\eta^{*}_{\pm})^{3/2}\Gamma_{\pm}^{3/2}} \\
         + O\left(\sqrt{\frac{P_t}{n}}\right)
    \end{multline*} 
    where $C \leq 0.4748$ is the Berry-Esseen constant.
\end{corollary}
Now we are ready to state the main result on anti-concentration, which follows from \cref{cor:berry_esseen}:
\begin{theorem}[Anti-concentration for Gaussian Kernel Estimator]\label{thm:anticonc:gaussian}
    Let $\Phi(x)$ be the CDF of the standard normal distribution and $\varepsilon>0$.
    Under the conditions of \cref{lem:anticoncentration_gaussian}, we have
    \begin{multline}\label{eq:anti_conc}
    \P(|\hat d(x,t) - d| \leq \varepsilon) \leq
    \Phi\left(\frac{(2^{\varepsilon/2}-1)(1+c)\sqrt{P_t n}}{2^{\varepsilon/2}\sqrt{\Gamma_+(1-\eta^*_+)}}\right) \\
    + \Phi\left(\frac{(1-2^{-\varepsilon/2})(1+c)\sqrt{P_t n}}{2^{-\varepsilon/2}\sqrt{\Gamma_-(1-\eta^*_-)}}\right) \\
    + \frac{C(1+\eta)2^{d/2}}{\sqrt{P_tn}}\left(\frac{\Delta_+}{(1-\eta^{*}_+)^{3/2}\Gamma_+^{3/2}} + \frac{\Delta_-}{(1-\eta^{*}_-)^{3/2}\Gamma_-^{3/2}}\right)
    + O\left(\sqrt{\frac{P_t}{n}}\right) - 1
    \end{multline}
    where $C \leq 0.4748$ is the Berry-Esseen constant, $c, \eta, \Gamma_\pm, \eta^*_\pm, \Delta_\pm$ are as defined in \cref{lem:anticoncentration_gaussian}, and $P_t$ is defined in \cref{eq:P_t}.
\end{theorem}

\begin{proof}
As in \cref{cor:berry_esseen}, define $Y_{-,i} = 2^{(d-\varepsilon)/2} K_t(x,X_i) - K_{2t}(x,X_i)$ with variance $\sigma_-^2 = \Var(Y_{-,i})$. 
From \cref{lem:anticoncentration_gaussian}, we have
\begin{equation*}
    \sigma_-^2 = \Var(Y_-) \geq 2^{-\varepsilon}\Gamma_-(1-\eta^*_-)P_t + O(P_t^2),
\end{equation*}
so $\sigma_- = 2^{-\varepsilon/2}\sqrt{\Gamma_-(1-\eta^*_-)P_t}(1 + O(P_t))$. 

Then beginning with the first term of decomposition in \cref{eq:gaussian:anti_decomp} and applying \cref{cor:berry_esseen}, we have 
\begin{align*}
    \P\left(2^{(d-\varepsilon)/2} S(x,t) - S(x,2t) \leq 0 \right) 
    &= \P\left( \frac{1}{\sqrt{n}} \sum_{i=1}^n \frac{Y_{-,i}-\E[Y_{-,i}]}{\sigma_-} \leq -\frac{\sqrt{n}\E[Y_{-,i}]}{\sigma_-}\right ) \\
    &\leq \Phi\left(\frac{\sqrt{n}\E[-Y_{-,i}]}{\sigma_-}\right) + \frac{C\rho_-}{\sigma_-^3\sqrt{n}} \\
    &\leq \Phi\left(\frac{\sqrt{n} \cdot (1-2^{-\varepsilon/2})(1+c)P_t}{2^{-\varepsilon/2}\sqrt{\Gamma_-(1-\eta^*_-)P_t}(1+O(P_t))}\right) \\
    &\phantom{=} + \frac{C(1+\eta)\Delta_-2^{d/2}}{(1-\eta^{*}_-)^{3/2}\Gamma_-^{3/2}}\frac{1}{\sqrt{nP_t}} + O\left(\sqrt{\frac{P_t}{n}}\right) \\
    &= \Phi\left(\frac{(1-2^{-\varepsilon/2})(1+c)\sqrt{P_t n}}{2^{-\varepsilon/2}\sqrt{\Gamma_-(1-\eta^*_-)}}\right) \\
    &\phantom{=} + \frac{C(1+\eta)\Delta_-2^{d/2}}{(1-\eta^{*}_-)^{3/2}\Gamma_-^{3/2}}\frac{1}{\sqrt{nP_t}} + O\left(\sqrt{\frac{P_t}{n}}\right).
\end{align*}

Similarly, let $Y_{+,i} = 2^{(d+\varepsilon)/2} K_t(x,X_i) - K_{2t}(x,X_i)$ with variance $\sigma_+^2 = \Var(Y_{+,i})$. From \cref{lem:anticoncentration_gaussian},
\begin{equation*}
    \sigma_+^2 = \Var(Y_+) \geq 2^{\varepsilon}\Gamma_+(1-\eta^*_+)P_t + O(P_t^2),
\end{equation*}
so $\sigma_+ = 2^{\varepsilon/2}\sqrt{\Gamma_+(1-\eta^*_+)P_t}(1 + O(P_t))$. Then for the second term of \cref{eq:gaussian:anti_decomp}, 
\begin{align*}
    \P\left(2^{(d+\varepsilon)/2} S(x,t) - S(x,2t) \geq 0 \right)
    &= \P\left( \frac{1}{\sqrt{n}} \sum_{i=1}^n \frac{Y_{+,i}-\E[Y_{+,i}]}{\sigma_+} \geq -\frac{\sqrt{n}\E[Y_{+,i}]}{\sigma_+}\right ) \\
    &\leq 1 - \Phi\left(-\frac{\sqrt{n}\E[Y_{+,i}]}{\sigma_+}\right) + \frac{C\rho_+}{\sigma_+^3\sqrt{n}} \\
    &= \Phi\left(\frac{\sqrt{n}\E[Y_{+,i}]}{\sigma_+}\right) + \frac{C\rho_+}{\sigma_+^3\sqrt{n}} \\
    &\leq \Phi\left(\frac{\sqrt{n} \cdot (2^{\varepsilon/2}-1)(1+c)P_t}{2^{\varepsilon/2}\sqrt{\Gamma_+(1-\eta^*_+)P_t}(1+O(P_t))}\right) \\
    &\phantom{=} + \frac{C(1+\eta)\Delta_+2^{d/2}}{(1-\eta^{*}_+)^{3/2}\Gamma_+^{3/2}}\frac{1}{\sqrt{nP_t}} + O\left(\sqrt{\frac{P_t}{n}}\right) \\
    &= \Phi\left(\frac{(2^{\varepsilon/2}-1)(1+c)\sqrt{P_t n}}{2^{\varepsilon/2}\sqrt{\Gamma_+(1-\eta^*_+)}}\right) \\
    &\phantom{=} + \frac{C(1+\eta)\Delta_+2^{d/2}}{(1-\eta^{*}_+)^{3/2}\Gamma_+^{3/2}}\frac{1}{\sqrt{nP_t}} + O\left(\sqrt{\frac{P_t}{n}}\right).
\end{align*}

Combining both bounds:
\begin{align*}
    \P(|\hat d(x,t) - d| \leq \varepsilon) &\leq \Phi\left(\frac{(2^{\varepsilon/2}-1)(1+c)\sqrt{P_t n}}{2^{\varepsilon/2}\sqrt{\Gamma_+(1-\eta^*_+)}}\right) \\
    &\phantom{=} + \Phi\left(\frac{(1-2^{-\varepsilon/2})(1+c)\sqrt{P_t n}}{2^{-\varepsilon/2}\sqrt{\Gamma_-(1-\eta^*_-)}}\right) \\
    &\phantom{=} + \frac{C(1+\eta)2^{d/2}}{\sqrt{P_tn}}\left(\frac{\Delta_+}{(1-\eta^{*}_+)^{3/2}\Gamma_+^{3/2}} + \frac{\Delta_-}{(1-\eta^{*}_-)^{3/2}\Gamma_-^{3/2}}\right) \\
    &\phantom{=} + O\left(\sqrt{\frac{P_t}{n}}\right) - 1.
\end{align*}
\end{proof}

\begin{remark}[Resolution regimes]\label{rem:regimes}
The bound in \cref{thm:anticonc:gaussian} reveals three regimes for the resolution $\varepsilon$, governed by the size of $P_t n$. Throughout, we need $\sqrt{P_t n} \gg 1$ for the error terms to be negligible (and in this regime, the $O(\sqrt{P_t/n})$ remainder is automatically small).
\begin{itemize}
    \item ($\varepsilon \ll 1/\sqrt{nP_t}$): The $\Phi$ terms vanish and the bound approaches zero, so accuracy $\varepsilon$ cannot be achieved.
    \item ($\varepsilon \sim 1/\sqrt{nP_t}$): This is the finest resolution attainable with $n$ samples at bandwidth $t$.
    \item ($\varepsilon \gg 1/\sqrt{nP_t}$): The $\Phi$ terms approach 1, the bound saturates, and the guarantee becomes vacuous.
\end{itemize}
Explicit constants in these scalings can be read off from the proof.
\end{remark}

The following remark makes the small-$\varepsilon$ scaling more explicit.

\begin{remark}[Small-$\varepsilon$ behavior]\label{rem:small_eps}
In the regime where $\varepsilon$ is small and $P_t n$ is large with $\varepsilon\sqrt{P_t n} \lesssim 1$, the bound in \cref{thm:anticonc:gaussian} simplifies. The $\Phi$ terms can be expanded as $\Phi(x) \approx \frac{1}{2} + \frac{x}{\sqrt{2\pi}}$ for small arguments, and the constants $\Gamma_\pm, \eta^*_\pm$ approach their $\varepsilon = 0$ limits:
\[
\Gamma_0 = 1 - 2\left(\frac{2}{3}\right)^{d/2} + 2^{-d/2}, \qquad
\eta_0^* = \eta\frac{1+2(2/3)^{d/2}+2^{-d/2}}{\Gamma_0}.
\]
The leading-order behavior becomes
\[
\P(|\hat d(x,t) - d| \leq \varepsilon) \approx \frac{\log 2}{\sqrt{2\pi\Gamma_0(1-\eta_0^*)}}(1+c)\sqrt{P_t n}\,\varepsilon.
\]
\end{remark}

\subsection{A bandwidth selection heuristic}\label{sec:bandwidth}

As noted in \cite{coifman2008}, to estimate the dimension we need to identify regions where the correct kernel statistics apply, where the kernel sum \cref{eq:kernel_sum} grows proportionally to $t^{d/2}$. \cref{fig:scaling_regimes} illustrates this scaling behavior numerically for the Gaussian kernel.

The bandwidth selection method in \cite{coifman2008} relied on visual inspection of log-log plots to find the most linear region, which \cite{BerryHarlim2016} automated via \emph{slope maximization} over a dyadic scale. We seek to improve upon this automatic selection criterion.

In many cases, slope maximization works well. However, even in the simple example of \cref{fig:scaling_regimes}, the maximal slope slightly overestimates the dimension. While the correct scaling regime lies near the maximal slope, the estimate is more stable within (rather than at the boundary of) this regime.

\begin{figure}[htbp]
    \centering
    \includegraphics[width=0.8\textwidth]{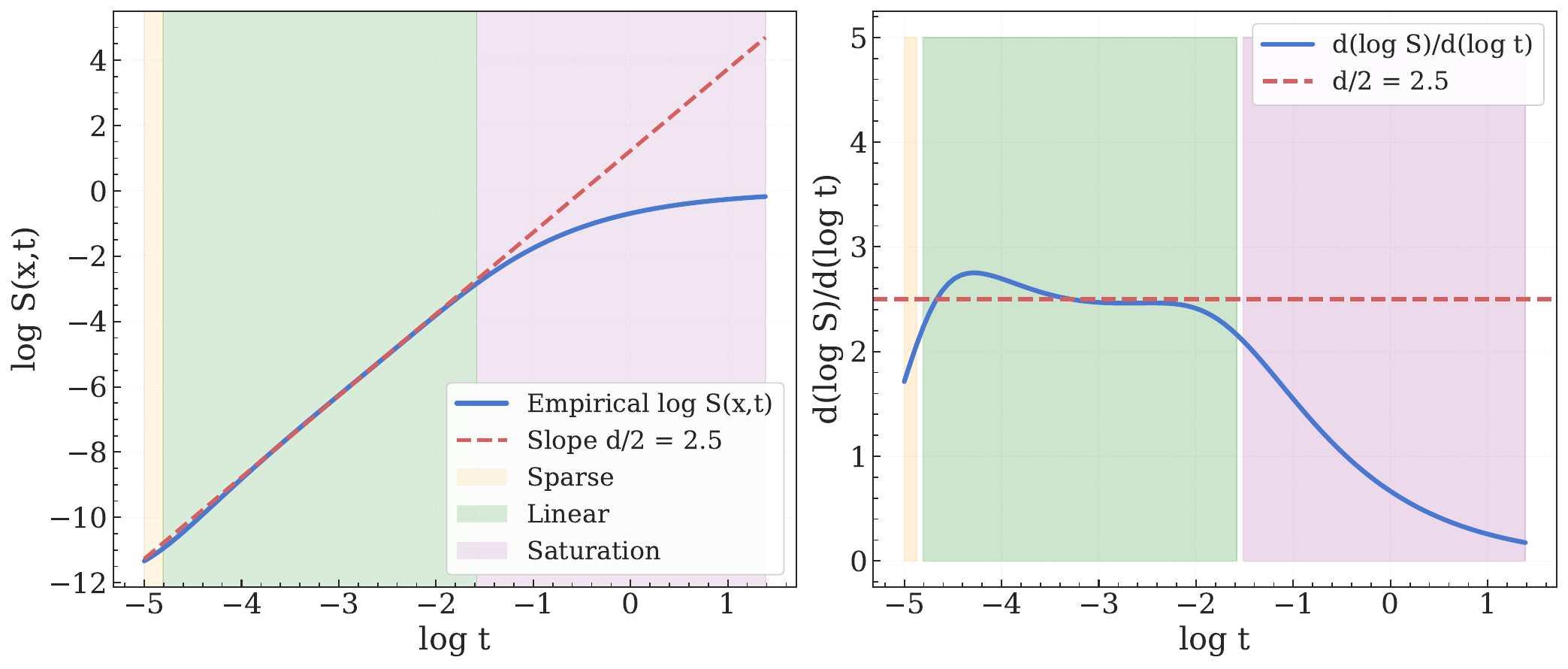}
    \caption{Scaling regimes of the kernel sum for unit ball in $\mathbb{R}^5$ with $n=100{,}000$ samples. Left: log-log plot showing linear regime with slope $d/2$. Right: derivative plateau indicating the linear regime.}
    \label{fig:scaling_regimes}
\end{figure}

One advantage of the Gaussian kernel over indicator kernels is smoothness: derivatives of the kernel sum are well-defined and can be computed efficiently (\cref{prop:derivatives}). This motivates exploring whether derivative information can improve bandwidth selection, an approach that is unavailable for discontinuous kernels such as the indicator kernel. We suggest utilizing more information in the derivatives of $\log S(x,t)$ to better capture the linear scaling region.

While a complete theoretical analysis of optimal bandwidth selection is beyond our scope, we propose a simple heuristic that can improve over slope maximization. If we define $G(t) \coloneq \log S(x,e^t)$, the heuristic is to maximize $\rho(t) = G'(t)/(|G''(t)| + \delta)$, which balances large slope with stability. Here $\delta$ is a small constant to avoid division by zero. The complete algorithm is summarized in \cref{alg:bandwidth_selection}, found in \cref{sec:bandwidth_proof}. In \cref{sec:numerical} we provide some numerical comparisons between this method and slope maximization.

\section{Numerical Experiments}\label{sec:numerical}
Having established our theoretical bounds, we now validate them numerically by comparing our bounds to empirical performance. We begin by applying our theoretical results \cref{thm:main,thm:anticonc:gaussian} to specific manifolds where we can calculate all necessary parameters.
 
After this we test our proposed bandwidth selection from \cref{sec:bandwidth}, by comparing our estimator against the maximization method of \cite{BerryHarlim2016}. We include the indicator kernel approach in our comparisons to provide a baseline.
\subsection{Concentration experiments}
We test \cref{thm:main} on two manifolds: a three-dimensional ball (flat geometry) and a three-dimensional spherical cap (curved), with parameters $L,M,r$ calculated similarly to \cref{ex:circle}.

In line with theory, in \cref{fig:concentration_comparison} our 
theoretical bounds successfully bound the empirical error, overestimating by approximately two orders of magnitude. Notably, the slopes indicate that our bounds are proportional to the correct scaling with respect to $n$.  
\begin{figure}[htbp]
    \centering
    \begin{subfigure}[b]{0.48\textwidth}
        \centering
        \includegraphics[width=\textwidth]{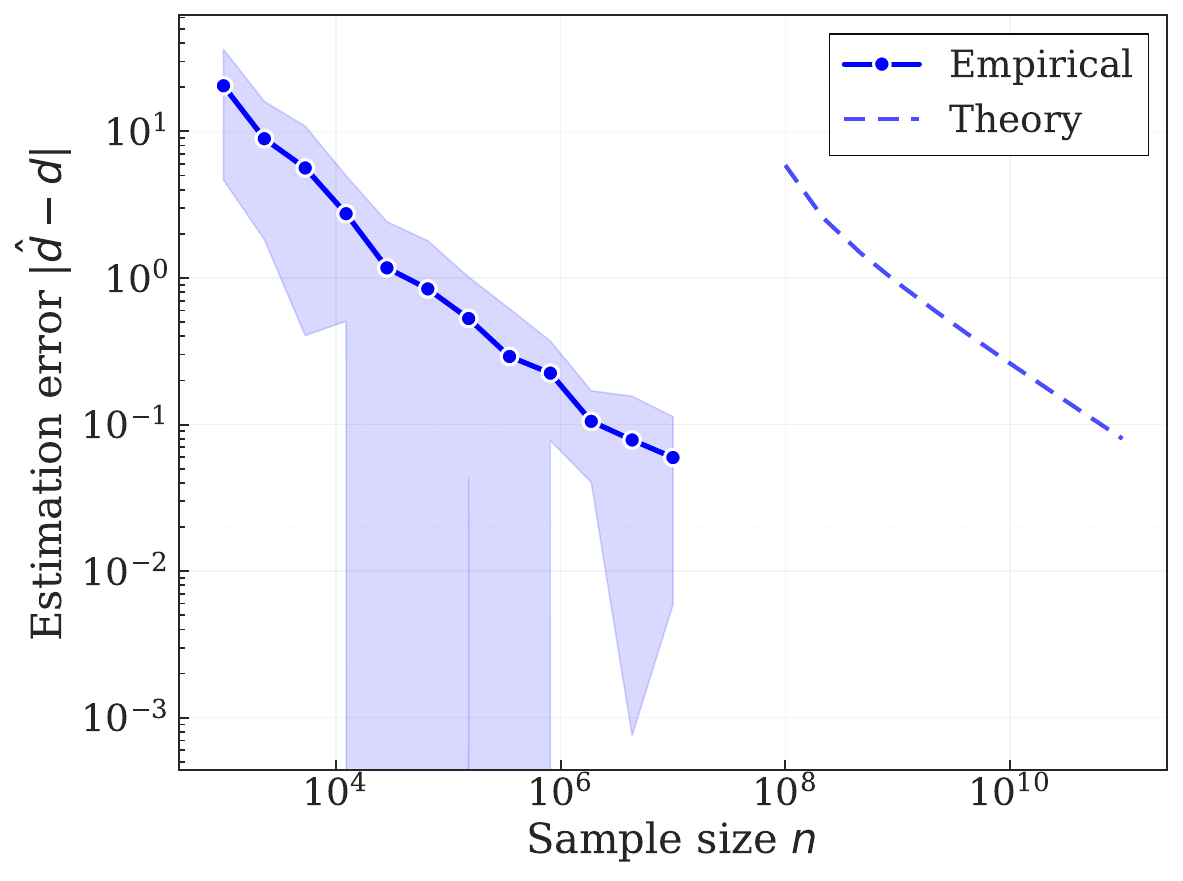}
        \caption{Flat manifold (ball)}
        \label{fig:concentration_ball}
    \end{subfigure}
    \hfill
    \begin{subfigure}[b]{0.48\textwidth}
        \centering
        \includegraphics[width=\textwidth]{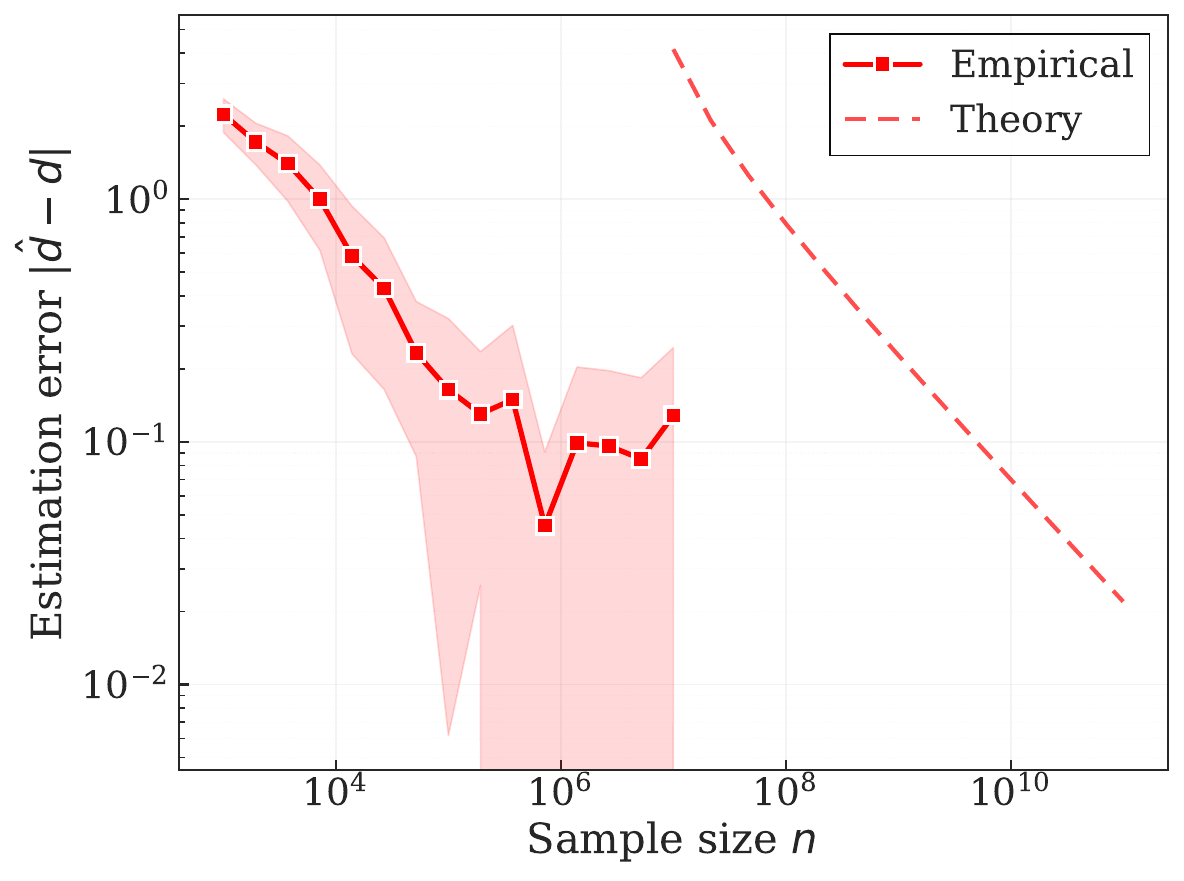}
        \caption{Curved manifold (spherical cap)}
        \label{fig:concentration_sphere}
    \end{subfigure}
    \caption{Concentration of dimension estimation error for intrinsic dimension $d=3$.
    Panel~(a) shows a uniform distribution on the unit ball in $\mathbb{R}^3$ (flat geometry) with bandwidth $t=0.000555 < t_0=0.000585$.
    Panel~(b) shows an approximately uniform distribution on a spherical cap of a 3-sphere with radius $R=10$ and bandwidth $t=0.000768 < t_0=0.000808$.
    Solid lines represent empirical mean errors, dashed lines show theoretical concentration bounds at 90\% confidence, and shaded regions indicate standard deviation. }
    \label{fig:concentration_comparison}
\end{figure}

In \cref{fig:anticoncentration} we have compared our theoretical anti-concentration bounds to the empirical ones for the three-dimensional spherical cap. In \cref{fig:anticonc_alpha_star} we show the linear scaling of \cref{rem:small_eps}, showing the accuracy limit for our estimator. For a given $n$, we 
cannot expect $\hat{d}$ to distinguish $d$ from nearby values within tolerance $\varepsilon^*$. In \cref{fig:anticonc_theory} we calculate our exact theoretical bounds for different values of $n$ and tolerances, which can be compared to empirical estimates in \cref{fig:anticonc_empirical}. As in the concentration setting, empirical performance exceeds the worst-case guarantees of \cref{thm:anticonc:gaussian}.

\begin{figure}[!htbp]
\centering
\begin{subfigure}{\textwidth}
    \centering
    \includegraphics[width=0.5\textwidth]{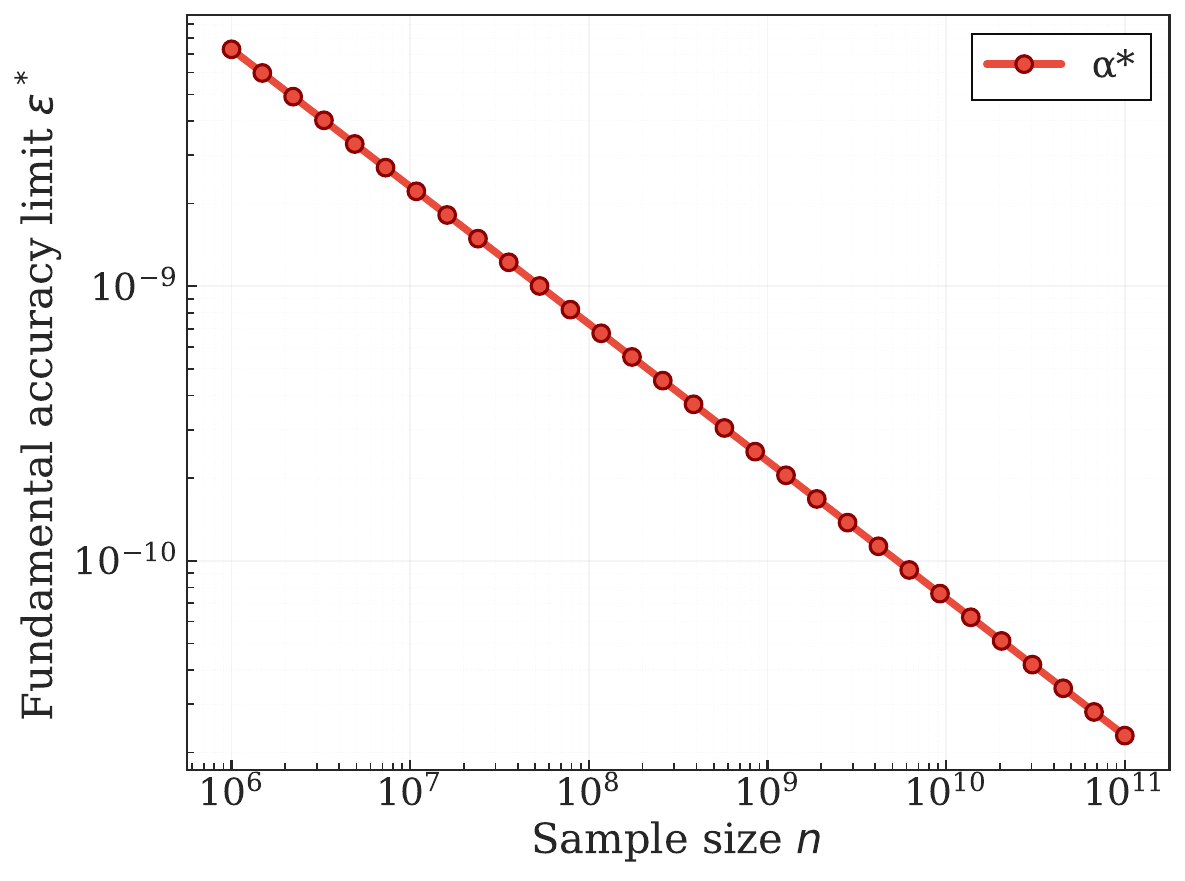}
    \caption{%
    Critical threshold $\varepsilon^*$ versus sample size $n$ from \cref{rem:small_eps}.
    Below $\varepsilon^*$, the upper bound on $\P(|\hat{d} - d| \leq \varepsilon)$ is at most $0.1$.
    }
    \label{fig:anticonc_alpha_star}
\end{subfigure}

\vspace{1em}

\begin{subfigure}[t]{0.48\textwidth}
    \centering
    \includegraphics[width=\textwidth]{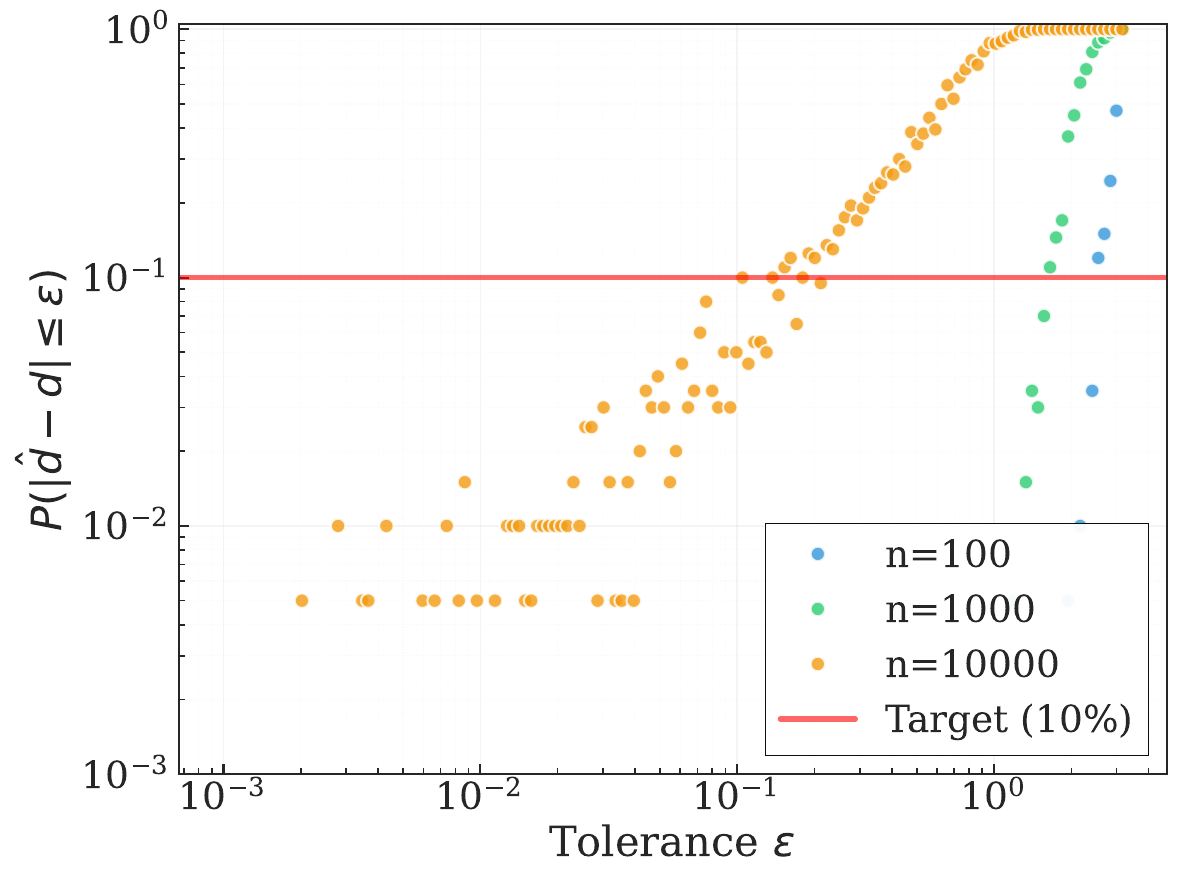}
    \caption{%
    Empirical probabilities $\P(|\hat{d} - d| \leq \varepsilon)$ over $200$ trials for $n \in \{100, 1000, 10000\}$.
    Red line: target threshold $0.1$.
    }
    \label{fig:anticonc_empirical}
\end{subfigure}
\hfill
\begin{subfigure}[t]{0.48\textwidth}
    \centering
    \includegraphics[width=\textwidth]{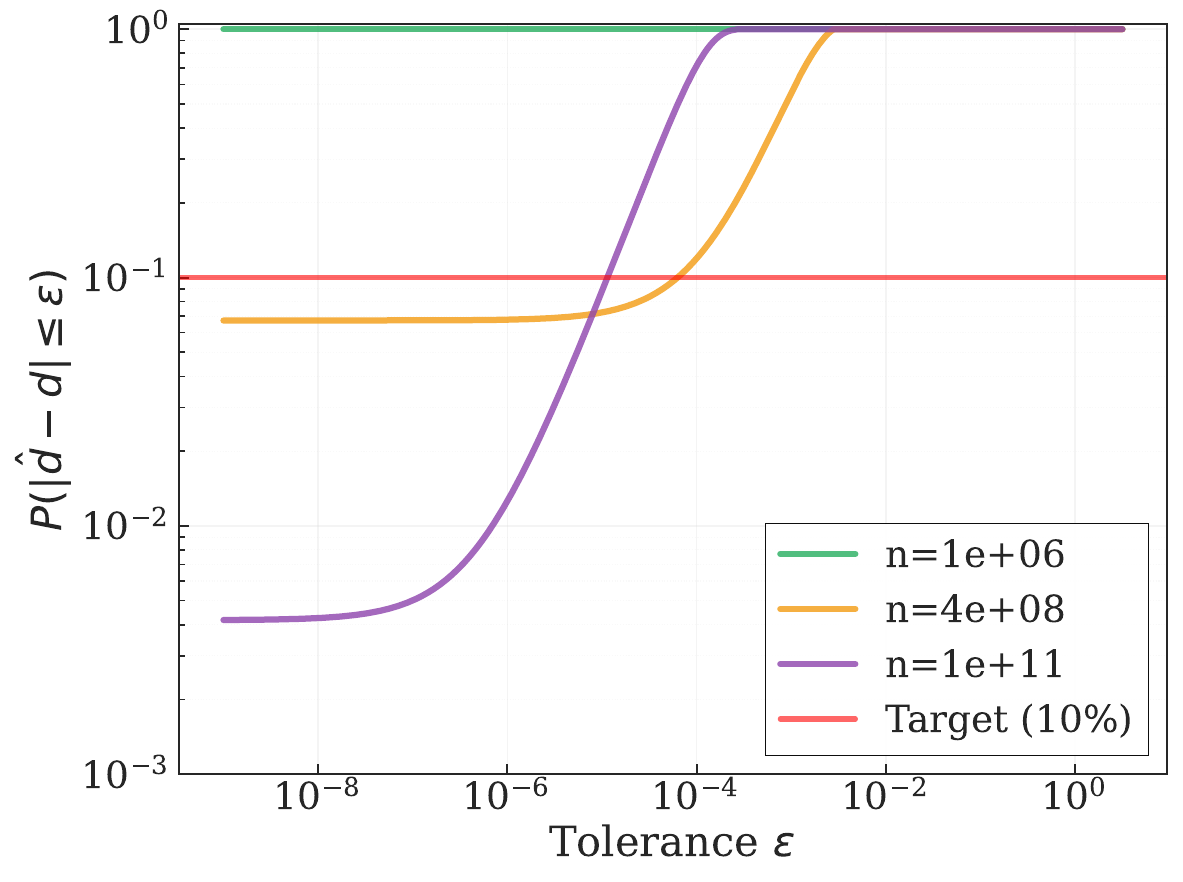}
    \caption{%
    Theoretical bounds from \cref{thm:anticonc:gaussian} for $n \in \{10^6, 10^8, 10^{11}\}$.
    }
    \label{fig:anticonc_theory}
\end{subfigure}

\caption{%
\textbf{Anti-concentration bounds on spherical cap.}
Experiments on a spherical cap of a 3-sphere with radius $R = 10$ (intrinsic dimension $d=3$, $L = 0.05$, $M = 0.005$) with bandwidth $t=0.000800 < t_0 = 0.000808$.
The anti-concentration bounds show fundamental accuracy limits.
}
\label{fig:anticoncentration}
\end{figure}

\subsection{Bandwidth selection experiments}
In this section we compare our bandwidth selection heuristic in \cref{alg:bandwidth_selection} to slope maximization, as well as the indicator kernel approach. Since noise is common in many applications, we add i.i.d. (ambient dimensional) Gaussian noise $\mathcal{N}(0, \sigma^2 I_N)$ to each of our samples before estimating 
dimension, for noise levels $\sigma \in \{0.0, 0.15, 0.30, 0.50\}$, and vary the number of samples. Each manifold is tested with a fixed reference point. The results are summarized in \cref{fig:method-comparison-manifolds} and \cref{fig:method-comparison-violin}.

We also extend our comparisons to the manifolds in the standard benchmark suggested in \cite{campadelli2015intrinsic}. However, we restrict our comparison to datasets with intrinsic dimension less than 10. This is because generally this class of estimators (more or less volume-based) do not perform well in higher dimensions, and comparisons are not interesting there. See \cref{fig:benchmark-difference} for a summary of this restricted benchmark.

\begin{figure}[H]
  \centering
  \includegraphics[width=\linewidth]{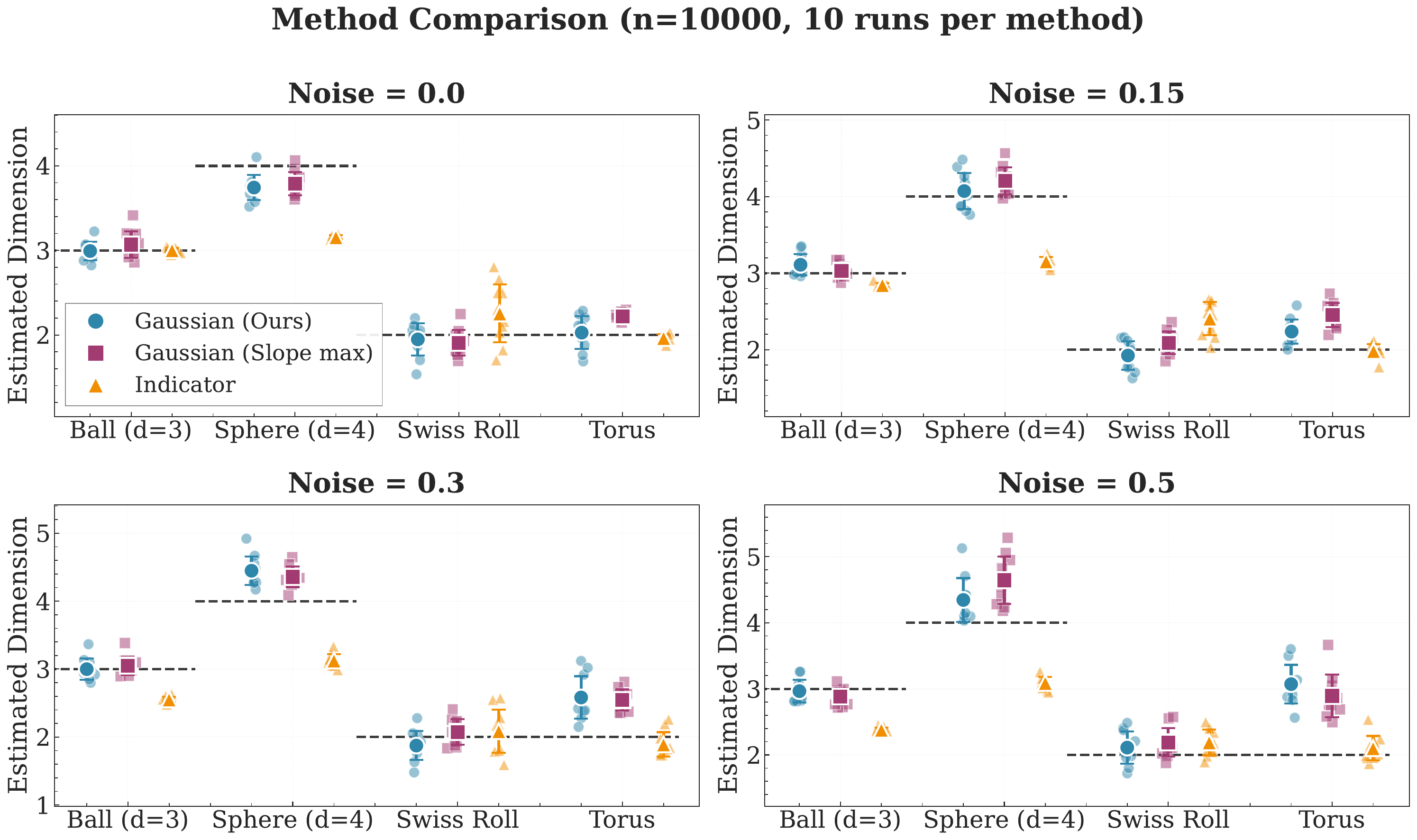}
  \caption{Distribution of estimated intrinsic dimension at noise levels $\sigma \in \{0.0, 0.15, 0.30, 0.50\}$ for the largest sample size ($n=10^4$). Each violin aggregates 10 runs on four manifolds (Ball $d=3$, Sphere $d=4$, Swiss Roll $d=2$, Torus $d=2$), showing Gaussian (Ours), Gaussian (Slope max), and indicator kernels side by side; the dashed line marks the true dimension.}
  \label{fig:method-comparison-violin}
\end{figure}

\begin{figure}[H]
  \centering
  \includegraphics[width=0.8\linewidth]{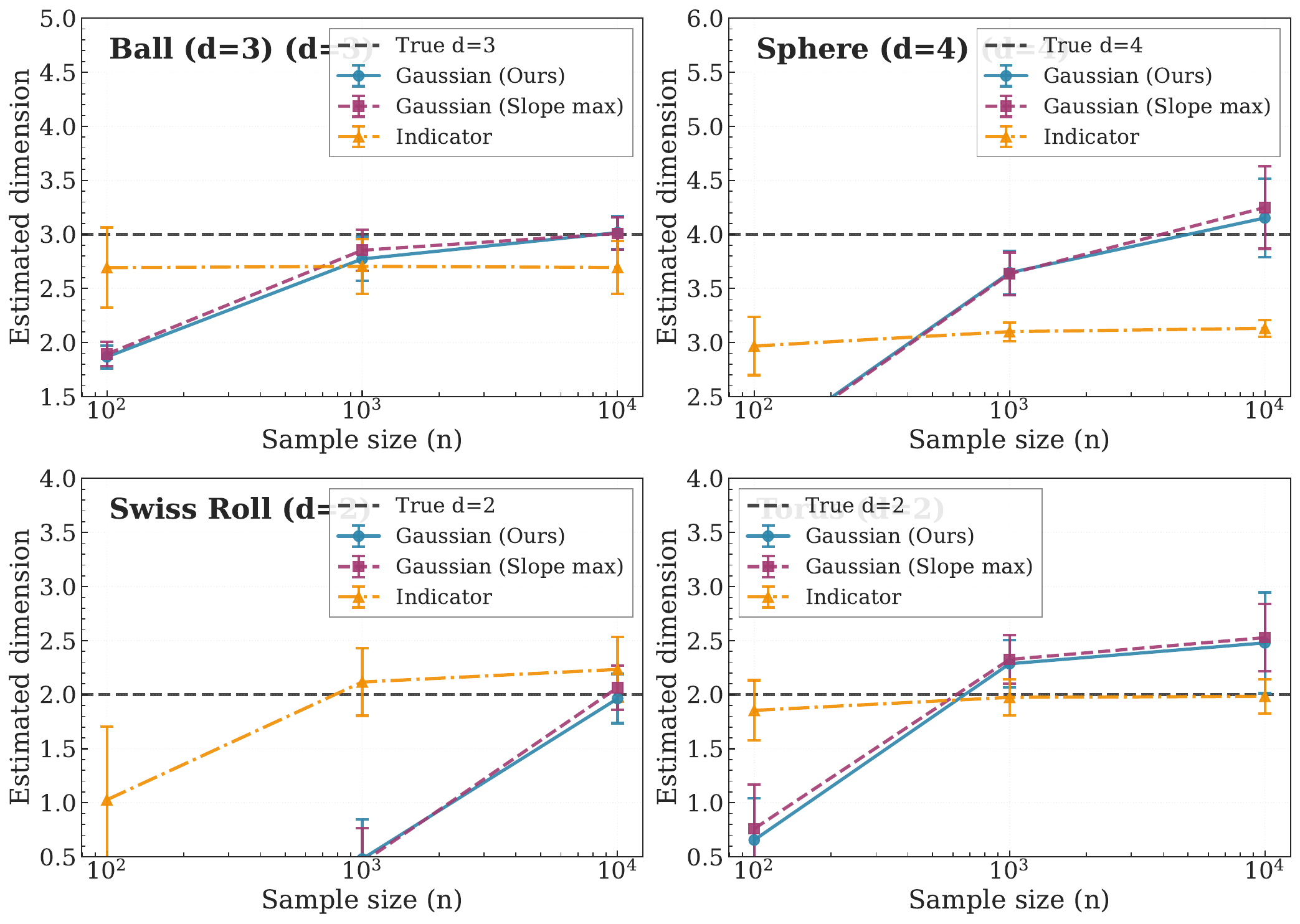}
  \caption{Gaussian (Ours) (blue), Gaussian (Slope max) (purple), and indicator kernel (orange) dimension estimates versus sample size ($n \in \{100, 1000, 10{,}000\}$) for Ball ($d=3$), Sphere ($d=4$), Swiss Roll ($d=2$), and Torus ($d=2$). Curves average 10 runs across all noise levels ($\sigma \in \{0.0, 0.15, 0.30, 0.50\}$), with error bars showing one standard deviation, and dashed lines marking the true intrinsic dimensions.}
  \label{fig:method-comparison-manifolds}
\end{figure}

\begin{figure}[H]
  \centering
  \includegraphics[width=0.9\linewidth]{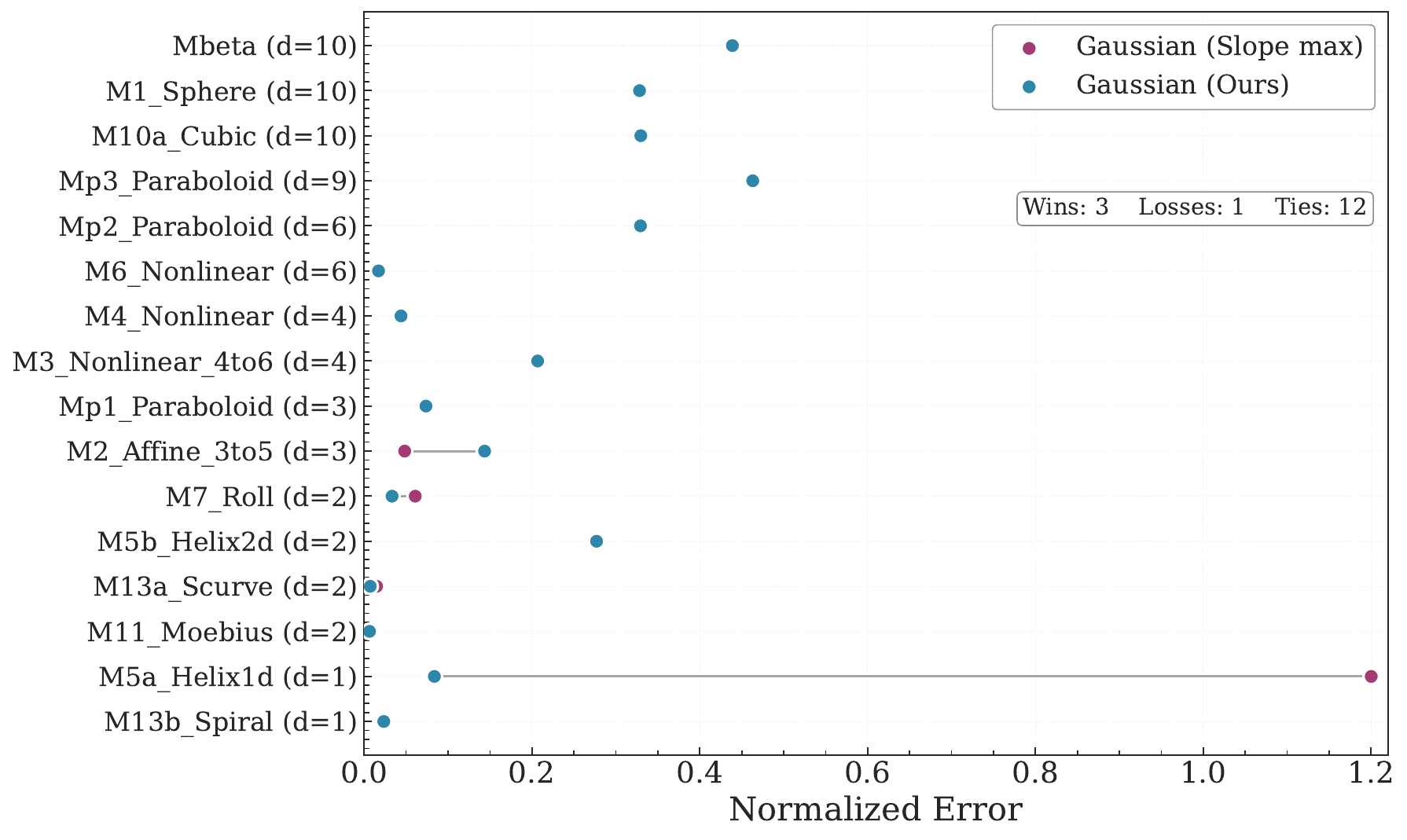}
    \caption{Normalized error $|\hat{d} - d|/d$ for Gaussian (Slope max) (purple) and Gaussian (Ours) (blue) across the scikit-dimension manifolds with intrinsic dimension $d \leq 10$ ($n=10{,}000$ samples, 10 runs per method, no ambient noise). Gray connectors link paired estimates; the inset reports how often our method wins, loses, or ties.}
  \label{fig:benchmark-difference}
\end{figure}

Performance varies across manifolds and noise levels, and no single method dominates uniformly. However, our proposed heuristic is the most robust overall, particularly under noise, while the indicator kernel is consistently less stable.

\section{Conclusion and future work}\label{sec:conclusion}
We have established concentration and anti-concentration properties of Gaussian kernel dimension estimators, with explicit dependence on sample size, bandwidth, and the geometric and distributional regularity of the underlying manifold. Our result in \cref{thm:main} provides finite-sample deviation bounds that can be used for rigorous confidence intervals and hypothesis tests. As a complement to this concentration result, the anti-concentration in \cref{thm:anticonc:gaussian} identifies fundamental limits on the precision achievable by our estimator. On the practical side, the bandwidth heuristic in \cref{alg:bandwidth_selection} leverages derivative information to improve empirical performance.

We mention a few directions for future work. Our concentration bounds overestimate compared with empirical errors, partly reflecting worst-case analysis with respect to $(L,M,r,p,\kappa)$ and intentional slack in the proofs; tightening these bounds is a natural next step. Relatedly, developing data-driven approaches to estimate these parameters would be valuable. Finally, the strong empirical performance of the bandwidth selection heuristic (\cref{alg:bandwidth_selection}, \cref{fig:method-comparison-violin,fig:method-comparison-manifolds,fig:benchmark-difference}) motivates further work on bandwidth selection using derivative information.
\section*{Acknowledgments}
The author was partly supported by the Wallenberg AI, Autonomous
Systems and Software Program (WASP) funded by the Knut and Alice Wallenberg Foundation.
\clearpage
\bibliography{bibliography}

@article{aamari2019estimating,
  title={Estimating the reach of a manifold},
  author={Aamari, Eddie and Kim, Jisu},
  journal={Electronic journal of statistics},
  volume={13},
  number={1},
  year={2019}
}

@article{amsaleg2018extreme,
  title={Extreme-value-theoretic estimation of local intrinsic dimensionality},
  author={Amsaleg, Laurent and Chelly, Oussama and Furon, Teddy and Girard, St{\'e}phane and Houle, Michael E and Kawarabayashi, Ken-ichi and Nett, Michael},
  journal={Data Mining and Knowledge Discovery},
  volume={32},
  number={6},
  pages={1768--1805},
  year={2018},
  publisher={Springer}
}

@article{ansuini2019intrinsic,
  title={Intrinsic dimension of data representations in deep neural networks},
  author={Ansuini, Alessio and Laio, Alessandro and Macke, Jakob H and Zoccolan, Davide},
  journal={Advances in Neural Information Processing Systems},
  volume={32},
  year={2019}
}

@article{belkin2003laplacian,
  title={Laplacian eigenmaps for dimensionality reduction and data representation},
  author={Belkin, Mikhail and Niyogi, Partha},
  journal={Neural computation},
  volume={15},
  number={6},
  pages={1373--1396},
  year={2003},
  publisher={MIT Press}
}

@article{BerryHarlim2016,
  title={Variable bandwidth diffusion kernels},
  author={Berry, Tyrus and Harlim, John},
  journal={Applied and Computational Harmonic Analysis},
  volume={40},
  number={1},
  pages={68--96},
  year={2016},
  publisher={Elsevier}
}

@article{BerryHarlim2018,
  title={Iterated diffusion maps for feature identification},
  author={Berry, Tyrus and Harlim, John},
  journal={Applied and Computational Harmonic Analysis},
  volume={45},
  number={1},
  pages={84--119},
  year={2018},
  publisher={Elsevier}
}

@article{binnie2025survey,
  title={A Survey of Dimension Estimation Methods},
  author={Binnie, James AD and Harvey, John and Malinowski, Jakub and Yim, Ka Man and others},
  journal={arXiv preprint arXiv:2507.13887},
  year={2025}
}

@article{birdal2021intrinsic,
  title={Intrinsic dimension, persistent homology and generalization in neural networks},
  author={Birdal, Tolga and Lou, Aaron and Guibas, Leonidas J and Simsekli, Umut},
  journal={Advances in neural information processing systems},
  volume={34},
  pages={6776--6789},
  year={2021}
}

@article{campadelli2015intrinsic,
  title={Intrinsic dimension estimation: Relevant techniques and a benchmark framework},
  author={Campadelli, Paola and Casiraghi, Elena and Ceruti, Claudio and Rozza, Alessandro},
  journal={Mathematical Problems in Engineering},
  volume={2015},
  year={2015},
  publisher={Hindawi}
}

@article{ceruti2012danco,
  title={DANCo: dimensionality from angle and norm concentration},
  author={Ceruti, Claudio and Bassis, Simone and Rozza, Alessandro and Lombardi, Gabriele and Casiraghi, Elena and Campadelli, Paola},
  journal={arXiv preprint arXiv:1206.3881},
  year={2012}
}

@article{coifman2006diffusion,
  title={Diffusion maps},
  author={Coifman, Ronald R and Lafon, St{\'e}phane},
  journal={Applied and computational harmonic analysis},
  volume={21},
  number={1},
  pages={5--30},
  year={2006},
  publisher={Elsevier}
}

@article{coifman2008,
  title={Graph Laplacian tomography from unknown random projections},
  author={Coifman, Ronald R and Shkolnisky, Yoel and Sigworth, Fred J and Singer, Amit},
  journal={IEEE Transactions on Image Processing},
  volume={17},
  number={10},
  pages={1891--1899},
  year={2008},
  publisher={IEEE}
}

@inproceedings{dupuis2023generalization,
  title={Generalization bounds using data-dependent fractal dimensions},
  author={Dupuis, Benjamin and Deligiannidis, George and Simsekli, Umut},
  booktitle={International Conference on Machine Learning},
  pages={8922--8968},
  year={2023},
  organization={PMLR}
}

@article{facco2017estimating,
  title={Estimating the intrinsic dimension of datasets by a minimal neighborhood information},
  author={Facco, Elena and d'Errico, Maria and Rodriguez, Alex and Laio, Alessandro},
  journal={Scientific reports},
  volume={7},
  number={1},
  pages={12140},
  year={2017},
  publisher={Nature Publishing Group UK London}
}

@article{fan2010intrinsic,
  title={Intrinsic dimension estimation of data by principal component analysis},
  author={Fan, Mingyu and Gu, Nannan and Qiao, Hong and Zhang, Bo},
  journal={arXiv preprint arXiv:1002.2050},
  year={2010}
}

@inproceedings{farahmand2007manifold,
  title={Manifold-adaptive dimension estimation},
  author={Farahmand, Amir Massoud and Szepesv{\'a}ri, Csaba and Audibert, Jean-Yves},
  booktitle={Proceedings of the 24th international conference on Machine learning},
  pages={265--272},
  year={2007}
}

@article{fukunaga1971algorithm,
  title={An algorithm for finding intrinsic dimensionality of data},
  author={Fukunaga, Keinosuke and Olsen, David R},
  journal={IEEE Transactions on computers},
  volume={100},
  number={2},
  pages={176--183},
  year={1971},
  publisher={IEEE}
}

@article{grassberger1983measuring,
  title={Measuring the strangeness of strange attractors},
  author={Grassberger, Peter and Procaccia, Itamar},
  journal={Physica D: nonlinear phenomena},
  volume={9},
  number={1-2},
  pages={189--208},
  year={1983},
  publisher={Elsevier}
}

@inproceedings{hein2005intrinsic,
  title={Intrinsic dimensionality estimation of submanifolds in Rd},
  author={Hein, Matthias and Audibert, Jean-Yves},
  booktitle={Proceedings of the 22nd international conference on Machine learning},
  pages={289--296},
  year={2005}
}

@article{levina2004maximum,
  title={Maximum likelihood estimation of intrinsic dimension},
  author={Levina, Elizaveta and Bickel, Peter},
  journal={Advances in neural information processing systems},
  volume={17},
  year={2004}
}

@article{pettis1979intrinsic,
  title={An intrinsic dimensionality estimator from near-neighbor information},
  author={Pettis, Karl W and Bailey, Thomas A and Jain, Anil K and Dubes, Richard C},
  journal={IEEE Transactions on pattern analysis and machine intelligence},
  number={1},
  pages={25--37},
  year={1979},
  publisher={IEEE}
}

@article{scholkopf1998nonlinear,
  title={Nonlinear component analysis as a kernel eigenvalue problem},
  author={Sch{\"o}lkopf, Bernhard and Smola, Alexander and M{\"u}ller, Klaus-Robert},
  journal={Neural computation},
  volume={10},
  number={5},
  pages={1299--1319},
  year={1998},
  publisher={MIT Press}
}

@article{tenenbaum2000global,
  title={A global geometric framework for nonlinear dimensionality reduction},
  author={Tenenbaum, Joshua B and Silva, Vin de and Langford, John C},
  journal={science},
  volume={290},
  number={5500},
  pages={2319--2323},
  year={2000},
  publisher={American Association for the Advancement of Science}
}

@article{von2007tutorial,
  title={A tutorial on spectral clustering},
  author={Von Luxburg, Ulrike},
  journal={Statistics and computing},
  volume={17},
  number={4},
  pages={395--416},
  year={2007},
  publisher={Springer}
}

@article{wang2004adaptive,
  title={Adaptive manifold learning},
  author={Wang, Jing and Zhang, Zhenyue and Zha, Hongyuan},
  journal={Advances in neural information processing systems},
  volume={17},
  year={2004}
}
\bibliographystyle{ieeetr}

\appendix
\section{Classical inequalities}\label{sec:app_classical}
For completeness we state the two classical results used in our proofs.
\begin{theorem}[Bernstein's inequality]\label{thm:bernstein}
    Let $X_1,\dots,X_n$ be independent random variables with finite variance such that $|X_i- \E[X_i]| \leq b$ for some $b>0$ almost surely for all $i\leq n$. Let
    \begin{equation*}
        S = \sum_{i=1}^n(X_i-\E[X_i]),
    \end{equation*}
    and $v=\sum_{i=1}^n \Var[X_i]$. Then
    \begin{equation*}
        \P\left(S\geq \varepsilon\right) \leq \exp\left(\frac{-\varepsilon^2}{2(v+b\varepsilon/3)}\right).
    \end{equation*}
\end{theorem}

\begin{theorem}[Berry--Esseen, i.i.d. with explicit constant]\label{thm:BE-iid}
Let \(X_1,X_2,\dots\) be i.i.d. with mean \(\mu\), variance \(\sigma^2>0\), and
\[
\rho:=\mathbb{E}|X_1-\mu|^{3}<\infty.
\]
Define \(S_n=\frac{1}{n}\sum_{i=1}^n X_i\) and
\[
F_n(x):=\mathbb{P}\Bigl(\frac{\sqrt{n}S_n-\sqrt{n}\mu}{\sigma}\le x\Bigr),\qquad
\Phi(x):=\frac{1}{\sqrt{2\pi}}\int_{-\infty}^{x}e^{-t^{2}/2}\,dt .
\]
Then
\[
\sup_{x\in\mathbb{R}}|F_n(x)-\Phi(x)|
\le
C\,\frac{\rho}{\sigma^{3}\sqrt{n}},
\qquad C\le 0.4748 .
\]
\end{theorem}

\section{Gaussian kernel proofs}\label{sec:app_main}
\subsection{Proof of \cref{lem:general}}
Since the manifold is locally $(L,M,r)$-regular at $x$, and $r_0 \leq r$, in the ball $B_r(x)$ we can use the projection $\pi$ from \cref{def:LMr-regular}. Let $y \in B_{r_0}(x) \cap \Omega$, $z \coloneqq \pi(y)-x \in T_x\Omega$ and $h_z \coloneqq  \pi(y)-y$. By $(L,M,r)$-regularity, we have that 
\begin{align}\label{eq:proj}
    \|x-y\|^2 = \|z\|^2 + \|h_z\|^2, \quad 0 \leq \|h_z\|^2 \leq L^2 \|z\|^4.
\end{align}
For the volume element of $\pi^{-1}$, $V(y)$, we have that
\begin{equation}\label{eq:dist}
    (1-M\|z\|^2) \leq V(y) \leq (1+M\|z\|^2),
\end{equation}
and since $p$ is $\kappa$-Lipschitz, we know that
\begin{equation}\label{eq:lip}
    |p(x)-p(y)| \leq \kappa\|x-y\| \leq \kappa(\|z\| + \|h_z\|) \leq 2\kappa r_0.
\end{equation}

To compute $\E[K_t(x,X)]$, we split the integration over $U\coloneq B_{r_0}(x) \cap \Omega$ and $U^c \coloneq \Omega \setminus B_{r_0}(x)$. For $y \in U^c$, $\|x-y\| \geq r_0$, since $p$ is a density, we have that
\begin{equation*}
    \int_{U^c} K_{t}(x,y)p(y)dy \leq e^{-r_0^2/t} = \tau(t)
\end{equation*}
Then using \cref{eq:proj,eq:dist,eq:lip} we get
\begin{equation}\label{eq:mom:lb}
    \begin{split}
     \int_\Omega K_{t}(x,y)p(y)\dy &\geq \int_{U} K_{t}(x,y)p(y)\dy \\
     &\geq \int_{U} e^{-(\|z\|^2+L^2\|z\|^4)/t}p(y)\dy \\
     &\geq e^{-L^2r_0^4/t}\int_{U} e^{-\|z\|^2/t}p(\pi^{-1}(z))V(y(z))\dz\\
     &\geq e^{-L^2r_0^4/t}(p(x)-2\kappa r_0)(1-Mr_0^2)\int_{\|z\| \leq r_0} e^{-\|z\|^2/t}\dz \\
     &= \left(1-\frac{2\kappa r_0}{p(x)}\right)(1-Mr_0^2)(1-\alpha(t))e^{-L^2r_0^4/t}p(x)\pi^{d/2}t^{d/2}
    \end{split}
\end{equation}
Similarly, for an upper bound we get
\begin{equation}\label{eq:mom:ub}
    \begin{split}
     \int_\Omega K_{t}(x,y)p(y)dy &\leq \int_U K_{t}(x,y)p(y)dy + \tau(t)\\
    &\leq (p(x)+2\kappa r_0) \int_{\|z\| \leq r_0} e^{-(\|z\|^2)/t}(1+M\|z\|^2)dz +  \tau(t) \\
    &\leq (p(x)+2\kappa r_0)(1+Mr_0^2)\int_{\R^d} e^{-\|z\|^2/t}dz +  \tau(t)\\
    &\leq(p(x)+2\kappa r_0)(1+Mr_0^2)\pi^{d/2}t^{d/2} + \tau(t) \\
    & = \left(1+\frac{2\kappa r_0}{p(x)}\right)(1+Mr_0^2)p(x)\pi^{d/2}t^{d/2} + \tau(t) \\
    \end{split}
\end{equation}
\subsection{Explicit formula for $t_0$ and proof of \cref{lem:mult}}
We first give the explicit formula for the threshold $t_0$ in \cref{lem:mult}. Define
\[
A \coloneqq -\log\left(\frac{\eta}{10} p(x)\pi^{d/2}2^{-d/2}\right),
\qquad
f(s)\coloneqq A s + \frac{d}{2(1-2\gamma)}s\log\frac{1}{s}, \quad s\in (0,\infty).
\]
Let $s_T = \inf\{s > 0 : f(s) \geq \tfrac12\}$, and set
$t_T \coloneqq s_T^{\,\frac{1}{\,1-2\gamma\,}}$. Then
\begin{multline}\label{eq:t0_explicit}
t_0(\eta,L,M,\kappa,\gamma,d,p) \\
 \;=\; \min\left\{
\left(\frac{\eta}{10M}\right)^{\!\frac{1}{2\gamma}},\;
\left(\frac{-\log(1-\eta/10)}{2L^2}\right)^{\!\frac{1}{4\gamma-1}},\;
t_T,\;
\left(\frac{p(x)\eta}{20\kappa}\right)^{\!\frac{1}{\gamma}},\;
\beta_{d,\eta}^{-\frac{1}{1/2-\gamma}}
\right\}.
\end{multline}
Here $\beta_{d,\eta} = \sqrt{2}\bigl(\sqrt{d}+\sqrt{2\log(10/\eta)}\bigr)$ is the Gaussian concentration threshold. The five terms control, respectively, volume distortion, curvature, tail decay, density variation, and Gaussian tail fraction. The function $f(s)$ is concave on $(0,\infty)$, so $s_T$ can be computed by finding $s^* = \argmax f$ via $f'(s^*)=0$; if $f(s^*) < 1/2$ then $s_T = +\infty$, otherwise bisect on $(0, s^*]$.

\begin{proof}[Proof of \cref{lem:mult}]
From \cref{eq:mom:general:lb,eq:mom:general:ub}, to get multiplicative bounds as in \cref{eq:mom:mult} we need to control the factors
\begin{equation*}
   \left(1-\frac{2\kappa r_0}{p(x)}\right)(1-Mr_0^2)(1-\alpha(\hat t))e^{-L^2r_0^4/\hat t}
\end{equation*}
for the lower bound (with four multiplicative factors: density variation, volume distortion, Gaussian tail fraction $\alpha$, curvature, and the leading $p(x)\pi^{d/2}\hat t^{d/2}$) and
\begin{equation*}
    \left(1+\frac{2\kappa r_0}{p(x)}\right)(1+Mr_0^2)\left(1+\frac{\tau(\hat t)}{p(x)\pi^{d/2}\hat t^{d/2}}\right)
\end{equation*}
for the upper bound. Note that the exponential factor $e^{-L^2r_0^4/\hat t} \leq 1$ does not appear in the upper bound.

If each of the four factors in the lower bound is at least $1-\eta/10$, and each of the three factors in the upper bound is at most $1+\eta/10$, then since $(1-\eta/10)^5 \geq 1-\eta$ for $\eta \in (0,1/2)$ and $(1+\eta/10)^3 \leq 1+\eta$, the result follows.

Now what follows is getting uniform bounds of the above for $\hat t \in \mathcal{H}$, finding the largest $t$ that satisfies them all. First we note that $r_0^2 = t^{2\gamma}$, and then that
\begin{equation*}
    \max_{\hat t\in \mathcal{H}} \frac{L^2r_0^4}{\hat t} = \frac{L^2r_0^4}{t/2} = 2L^2t^{4\gamma -1}, \qquad \max_{\hat t \in \mathcal{H}}e^{-\frac{r_0^2}{\hat t}} = e^{-\frac{r_0^2}{2t}} = e^{-\frac{1}{2}t^{2\gamma -1}}. 
\end{equation*}
 From this we arrive at the following sufficient conditions:
\begin{equation*}
    \begin{split}
    & \quad \frac{2\kappa t^{\gamma}}{p(x)} \leq \frac{\eta}{10} \\
    & \quad Mt^{2\gamma} \leq \frac{\eta}{10} \\
    & \quad 2L^2t^{4\gamma -1} \leq -\log(1-\eta/10) \\
    & \quad e^{-\frac{1}{2}t^{2\gamma-1}} \leq  \frac{\eta}{10} p(x)\pi^{d/2}\left(\frac{t}{2}\right)^{d/2} \\
    & \quad \alpha(2t) \leq \frac{\eta}{10}
    \end{split}
\end{equation*}
The fifth condition controls the Gaussian tail fraction $\alpha(\hat t)$. The worst case over $\hat t \in \mathcal{H}$ is $\hat t = 2t$ (widest kernel relative to fixed $r_0$). The condition $\alpha(2t) \leq \eta/10$ is equivalent to
\[
\P\!\left(\|Z\| > r_0/\sqrt{2t}\right) \leq \eta/10, \qquad Z\sim\mathcal{N}(0,I_d).
\]
By Gaussian concentration of Lipschitz functions, $\P(\|Z\| > \sqrt{d}+s) \leq e^{-s^2/2}$. So a sufficient condition is
\[
\frac{r_0}{\sqrt{2t}} = \frac{t^{\gamma-1/2}}{\sqrt{2}} \geq \sqrt{d} + \sqrt{2\log(10/\eta)}.
\]
Equivalently $t \leq \beta_{d,\eta}^{-1/(1/2-\gamma)}$ where $\beta_{d,\eta} = \sqrt{2}(\sqrt{d}+\sqrt{2\log(10/\eta)})$.

For clarity, we reformulate the fourth condition as follows:
\begin{align*}
    e^{-\frac{1}{2}t^{2\gamma-1}} &\leq  \frac{\eta}{10} p(x)\pi^{d/2}\left(\frac{1}{2}\right)^{d/2}t^{d/2} \\
    -\frac{1}{2}t^{2\gamma-1} & \leq \log\left(\frac{\eta}{10} p(x)\pi^{d/2}2^{-d/2}\right) - \frac{d}{2}\log \frac{1}{t} \\
    \frac{1}{2}t^{2\gamma-1} & \geq -\log\left(\frac{\eta}{10} p(x)\pi^{d/2}2^{-d/2}\right) + \frac{d}{2}\log \frac{1}{t}
\end{align*}
With $s \coloneq t^{1-2\gamma}$ and $A = -\log\left(\frac{\eta}{10} p(x)\pi^{d/2}2^{-d/2}\right)$, the above implies
\begin{equation*}
    \frac{1}{2s} \geq A + \frac{d}{2(1-2\gamma)}\log \frac{1}{s} \iff
    \frac{1}{2}  \geq As + \frac{d}{2(1-2\gamma)} s \log \frac{1}{s}.
\end{equation*}
The expression on the right is concave in $s$ over $[0,\infty)$, which implies we can find $s_T = \inf\{s > 0 : f(s) \geq \tfrac12\}$.
\end{proof}

\subsection{Proof of \cref{cor:w}}
    By \cref{lem:mult}
    \begin{align*}
        \E[W^+] &=\E[K_{2t}(x,X)] - 2^{(d+\varepsilon)/2}\E[K_t(x,X)] \\
        &\leq (1+\eta)P_t - (1-\eta)2^{\varepsilon/2}P_t  \\
        &\leq (1+\eta - 2^{\varepsilon/2}(1-\eta))P_t = (1-2^{\varepsilon/2}+\eta(1+2^{\varepsilon/2}))P_t \\
        &\leq (1-c)(1-2^{\varepsilon/2})P_t.
    \end{align*}
    Similarly,
    \begin{align*}
        \E[W^-] &=\E[K_{2t}(x,X)] - 2^{(d-\varepsilon)/2}\E[K_t(x,X)] \\
        &\geq (1-\eta)P_t - (1+\eta)2^{-\varepsilon/2}P_t  \\
        &\geq (1-\eta - 2^{-\varepsilon/2}(1+\eta))P_t = (1-2^{-\varepsilon/2}-\eta(1+2^{-\varepsilon/2}))P_t \\
        & \geq (1-c)(1-2^{-\varepsilon/2})P_t.
    \end{align*}

    For the variance, first we use that $\Var(W) \leq \E[W^2]$. Then since $K_t^2 = K_{t/2}, K_{2t}^2 = K_t, K_tK_{2t} = K_{2t/3}$,
    \begin{align*}
        \Var(W^{+}) &\leq\E[\left(W^{+}\right)^2] \leq\left(\E\left[K_{2t}^2 + 2^{(d+\varepsilon)/2+1}K_tK_{2t} + 2^{d+\varepsilon}K_{t}^2\right]\right) \\
        & =\left(\E \left[ K_{t}+2^{(d+\varepsilon)/2+1}K_{2t/3}+2^{d+\varepsilon}K_{t/2}\right]\right) \\
        &\leq (1+\eta)\left(1+2^{(d+\varepsilon)/2+1}\left(\frac{2}{3}\right)^{d/2}+2^{d/2 + \varepsilon}\right)\pi^{d/2}t^{d/2}p(x) \\
        &\leq \left(1+1/2\right)\left(2^{-d/2-\varepsilon}+2^{-\varepsilon/2+1}\left(\frac{2}{3}\right)^{d/2}+1\right)2^{\varepsilon}P_t \\
        &\leq 2^{\varepsilon+3}P_t \\
    \end{align*}
Similarly for $W^-$, but we have instead $\Var(W^-) \leq 2^{-\varepsilon+3}P_t$.
\subsection{Proof of \cref{lem:anticoncentration_gaussian}} 
    Throughout, we use the multiplicative bounds from \cref{lem:mult}, which hold for $t$ sufficiently small.
    \subsubsection*{First moment bounds}
    Using the definition of $\eta$ and \cref{lem:mult}, direct calculation gives
    \begin{equation}\label{eq:y+_exp}
        \begin{split}
        \E[Y_+] &\leq \left((1+\eta)2^{\varepsilon/2}-(1-\eta)\right)P_{t} = \left(2^{\varepsilon/2}-1 + \eta (2^{\varepsilon/2}+1)\right)P_t \\
        &=(2^{\varepsilon/2}-1)(1+c)P_t
        \end{split}
    \end{equation}
    and
    \begin{align*}
         \E[-Y_-] &\leq \left((1+\eta)-(1-\eta)2^{-\varepsilon/2}\right)P_{t} = \left(1-2^{-\varepsilon/2}+ \eta\left(2^{-\varepsilon/2}+1\right) \right)P_t \\
         & = \left(1-2^{-\varepsilon/2}\right)(1+c)P_t
    \end{align*}

    \subsubsection*{Variance bounds}
    Using \cref{eq:mom},
    \begin{equation}\label{eq:y_sq}
        \begin{split}
       \E[Y_+^2] &\geq\E\left[(1-\eta)2^{d+\varepsilon}K_{t/2} - (1+\eta)2^{(d+\varepsilon)/2+1}K_{2t/3}+(1-\eta)K_t\right] \\
        &=\left((1-\eta)(2^{d/2+\varepsilon}+1)-(1+\eta)(2^{d+\varepsilon/2+1}3^{-d/2})\right)\pi^{d/2}t^{d/2}p(x) \\
        & = \left(\left(2^{d/2+\varepsilon}-2^{d+\varepsilon/2+1}3^{-d/2}+1\right)-\eta\left(2^{d/2+\varepsilon}+1+2^{d+\varepsilon/2+1}3^{-d/2}\right)\right)\pi^{d/2}t^{d/2}p(x)  \\
        & = \left(\left(2^{\varepsilon}-2^{d/2+\varepsilon/2+1}3^{-d/2}+2^{-d/2}\right)-\eta\left(2^{\varepsilon}+2^{d/2+\varepsilon/2+1}3^{-d/2}+2^{-d/2}\right)\right)P_t \\
        & = 2^{\varepsilon}P_t\left(\left(1-2^{1-\varepsilon/2}\left(\frac{2}{3}\right)^{d/2}+2^{-d/2-\varepsilon}\right) - \eta\left(1+2^{1-\varepsilon/2}\left(\frac{2}{3}\right)^{d/2}+2^{-d/2-\varepsilon}\right)\right) \\
        &= 2^{\varepsilon}\Gamma_+(1-\eta^{*}_+)P_t,
        \end{split}
    \end{equation}
    The condition $\eta^*_+ < 1$ ensures this bound is positive. From \cref{eq:y+_exp}
    \begin{align*}
        \E[Y_+]^2 \leq \left((2^{\varepsilon/2}-1)(1+c)P_t\right)^2.
    \end{align*}
    Then
    \begin{equation*}
        \begin{split}
        \Var(Y_+) &=\E[Y_+^2] - \E[Y_+]^2 \geq  2^{\varepsilon}\Gamma_+(1-\eta^{*}_+)P_t - \left((2^{\varepsilon/2}-1)(1+c)P_t\right)^2 \\
        &= 2^{\varepsilon}\Gamma_+(1-\eta^{*}_+)P_t + \bigO(P_t^2) = C_\sigma P_t + \bigO(P_t^2),
        \end{split}
    \end{equation*}
     where $C_\sigma = 2^{\varepsilon}\Gamma_+(1-\eta^{*}_+)$.
    \subsubsection*{Third moment bounds}
    First we note that using \cref{eq:mom}, similar calculations as in \cref{eq:y_sq} yield $\E[Y_+^m] = \bigO(P_t)$ and \cref{eq:y+_exp} gives $\E[Y_+]^m = \bigO(P_t^m)$. Thus,
    \begin{align*}
        \E[|Y_+-\E[Y_+]|^3] \leq \E[|Y_+|^3] + \bigO(P_t^2).
    \end{align*}
    We proceed to bound $\E[|Y_+|^3]$. For $t$ small enough that \cref{lem:mult} extends to $\hat t = 2t/5$ (via \cref{lem:general}), we have
    \begin{align*}
        \E[|Y_+|^3] &\leq \E\left[2^{3(d+\varepsilon)/2}K_{t}^3 + 3\cdot 2^{d+\varepsilon}K_t^2 K_{2t} + 3\cdot 2^{(d+\varepsilon)/2}K_tK_{2t}^2 + K_{2t}^3 \right] \\
        &=\E\left[2^{3(d+\varepsilon)/2}K_{t/3}+3\cdot 2^{d+\varepsilon}K_{2t/5}+ 3\cdot 2^{(d+\varepsilon)/2}K_{t/2}+ K_{2t/3} \right] \\
        &\leq (1+\eta)P_t\left(2^{d+3\varepsilon/2}3^{-d/2} + 3\cdot 2^{d+\varepsilon}5^{-d/2} +3\cdot 2^{\varepsilon/2-d/2}+3^{-d/2}\right) \\
        &  = (1+\eta)P_t2^{d/2}2^{3\varepsilon/2}\left[\left(\frac{2}{3}\right)^{d/2} + 3 \cdot 2^{d/2-\varepsilon/2}5^{-d/2} + 3 \cdot 2^{-\varepsilon-d} + 2^{-3\varepsilon/2-d/2} 3^{-d/2}\right] \\
        & = (1+\eta)\Delta_+ 2^{d/2}2^{3\varepsilon/2}P_t = C_\rho P_t,
    \end{align*}
    where $C_\rho =(1+\eta)\Delta_+ 2^{d/2}2^{3\varepsilon/2}$. 

    The bounds for $Y_-$ follow by replacing $\varepsilon \to -\varepsilon$ throughout.
\subsection{Proof of \cref{cor:berry_esseen}}
    We prove the result for $Y_{+,i} = 2^{(d+\varepsilon)/2}K_t(x,X_i) - K_{2t}(x,X_i)$; the proof for $Y_{-,i} = 2^{(d-\varepsilon)/2}K_t(x,X_i) - K_{2t}(x,X_i)$ is identical with $+\varepsilon/2$ replaced by $-\varepsilon/2$ throughout.
    
    If we apply the Berry-Esseen to the normalized sum $\frac{\sum_{i=1}^n (Y_{i,\pm} - \E[Y_{i,\pm}])}{\sqrt{n\Var(Y_{i,\pm})}}$, by Berry-Esseen, the approximation error is bounded by $C\frac{\rho_+}{\sigma_+^3\sqrt{n}}$,
    where $\sigma_+^2 = \Var(Y_{+,i})$, $\rho_+ = \mathbb{E}[|Y_{+,i} - \mathbb{E}[Y_{+,i}]|^3]$, and $C \leq 0.4748$.
    
    From \cref{lem:anticoncentration_gaussian}, we have
    \begin{align*}
        \sigma_+^2 &= \Var(Y_+) \geq 2^{\varepsilon}\Gamma_+(1-\eta^*_+)P_t + \bigO(P_t^2) =: C_\sigma P_t + \bigO(P_t^2), \\
        \rho_+ &\leq (1+\eta)\Delta_+2^{d/2}2^{3\varepsilon/2}P_t + \bigO(P_t^2) =: C_\rho P_t + \bigO(P_t^2),
    \end{align*}
    where $C_\sigma = 2^{\varepsilon}\Gamma_+(1-\eta^*_+)$ and $C_\rho = (1+\eta)\Delta_+2^{d/2}2^{3\varepsilon/2}$.
    
    Using Taylor expansion, we see that
    \begin{equation*}
        \begin{split}
        \frac{\rho}{\sigma^3\sqrt{n}} &\leq \frac{C_\rho P_t + \bigO(P_t^2)}{\sqrt{n}(C_\sigma P_t +\bigO(P_t^2))^{3/2}} = \frac{C_\rho P_t(1+\bigO(P_t))}{\sqrt{n}(C_\sigma P_t)^{3/2}(1+\bigO(P_t))^{3/2}} \\
        &= \frac{C_\rho}{C_\sigma^{3/2}\sqrt{P_t n}} \cdot \frac{1+\bigO(P_t)}{(1+\bigO(P_t))^{3/2}} = \frac{C_\rho}{C_\sigma^{3/2}\sqrt{P_t n}}(1+\bigO(P_t)) \\
        &= \frac{C_\rho}{C_\sigma^{3/2}\sqrt{n}}\left(\frac{1}{\sqrt{P_t}}+\bigO(\sqrt{P_t})\right).
        \end{split}
    \end{equation*}
    The expansion was made on $(1+\bigO(P_t))(1+\bigO(P_t))^{-3/2}$ in the second line above.
    
    Further,
    \begin{equation*}
        \begin{split}
        \frac{C_\rho}{C_\sigma^{3/2}} = \frac{(1+\eta)\Delta_+2^{d/2}2^{3\varepsilon/2}}{\left(2^{\varepsilon}\Gamma_+(1-\eta^{*}_+)\right)^{3/2}} = \frac{(1+\eta)\Delta_+2^{d/2}}{(1-\eta^{*}_+)^{3/2}\Gamma_+^{3/2}}
        \end{split}
    \end{equation*}

    Therefore, the normal approximation error for $\frac{1}{n}\sum_{i=1}^n Y_{i,+}$ is
    \begin{equation*}
        \frac{C(1+\eta)\Delta_+2^{d/2}}{\sqrt{n}\sqrt{P_t}(1-\eta^{*}_+)^{3/2}\Gamma_+^{3/2}} + O\left(\sqrt{\frac{P_t}{n}}\right)
    \end{equation*}
    where $C \leq 0.4748$. The proof for $Y_-$ is identical.
\section{Bandwidth}\label{sec:bandwidth_proof}
Recall $G(t) \coloneq \log S(x,e^t)$.
\begin{algorithm}[H]
\caption{Bandwidth Selection via Curvature Analysis}
\label{alg:bandwidth_selection}
\begin{algorithmic}[1]
\Require Samples $x_1,\dots,x_n \in \R^N$, reference point $x_0 \in \R^N$
\Ensure Optimal bandwidth $t^*$, dimension estimate $\hat{d}$
\State Choose a log-spaced grid of bandwidths $\{t_j\} \subset [t_{\min}, t_{\max}]$
\State Compute $S_j = S(x_0,t_j)=\frac{1}{n}\sum_{i=1}^n K_{t_j}(x_0,x_i)$ for all $j$
\State For each $t_j$, compute $G'_j$ and $G''_j$ using \cref{prop:derivatives}
\State Define $\rho_j = G'_j/(|G''_j|+\delta)$ with $\delta=10^{-3}$
\State $t^* \gets \argmax_j \rho_j$
\State $\hat{d} \gets 2\log_2\!\big(S(x_0,2t^*)/S(x_0,t^*)\big)$
\State \Return $t^*, \hat{d}$
\end{algorithmic}
\end{algorithm}

\begin{proposition}\label{prop:derivatives}
If we define
\[
\langle f \rangle = \frac{\sum_{i=1}^n f(u_i) e^{-u_i}}{\sum_{i=1}^n e^{-u_i}}, \quad u_i = \frac{\|x-X_i\|^2}{e^t},
\]
then:
\begin{align*}
    G'(t) &= \langle u_i \rangle, \\
    G''(t) &= \langle u_i^2 - u_i \rangle - \langle u_i \rangle^2.
\end{align*}
\end{proposition}
\begin{proof}
Set $u_i(t) = \|x-X_i\|^2/e^t$ and $K_i(t) = e^{-u_i(t)}$. Write $S = S(x,e^t) = \frac{1}{n}\sum_{i=1}^n K_i$ and define the weighted average $\langle f \rangle = \frac{1}{nS}\sum_{i=1}^n f(u_i)K_i$.

Direct calculation gives:
\begin{align*}
    u_i' &= -u_i, \\
    K_i' &= u_iK_i, \\
    K_i'' &= u_i(u_i-1)K_i.
\end{align*}

Therefore:
\begin{align*}
    G'(t) &= \frac{S'}{S} = \langle u_i \rangle, \\
    G''(t) &= \frac{S''}{S} - \left(\frac{S'}{S}\right)^2 = \langle u_i^2 - u_i \rangle - \langle u_i \rangle^2.
\end{align*}
\end{proof}
\end{document}